\documentclass[final,leqno]{siamltex}
\usepackage{booktabs}
\usepackage[noadjust]{cite}
\usepackage{amsmath}
\allowdisplaybreaks[4]
\usepackage{multirow}
\usepackage{graphicx}
\usepackage{placeins}
\usepackage{amssymb}
\usepackage{mathtools}
\usepackage{mathrsfs}
\newcommand{\norm}[1]{\left\|#1\right\|}
\usepackage{latexsym, bm}
\usepackage{hyperref}
\numberwithin{equation}{section}
\newtheorem{assumption}[theorem]{Assumption}
\newtheorem{remark}{Remark}[section]

\newcommand{\vertiii}[1]{{\left\vert\kern-0.23ex\left\vert\kern-0.23ex\left\vert #1
		\right\vert\kern-0.23ex\right\vert\kern-0.23ex\right\vert}}
\renewcommand\abstractname{Abstract}

\renewcommand\figurename{Figure}
\renewcommand\tablename{Table}

\normalsize

\renewcommand\refname{References}

\def\ba{\begin{equation}\begin{aligned}}
\def\ed{\end{aligned}\end{equation}}

\newcommand{\R}{\mathbb{R}}
\newcommand{\E}{\mathcal{E}}
\newcommand{\Lop}{\mathcal{L}}

\newcommand{\Aset}{\mathcal{A}}
\newcommand{\Ih}{\mathcal{I}_h}
\newcommand{\inner}[2]{\left\langle #1,#2\right\rangle}
\newcommand{\innerh}[2]{\left\langle #1,#2\right\rangle_h}
\newcommand{\normh}[1]{\left\|#1\right\|_h}
\newcommand{\dt}{\delta_t}
\newcommand{\du}{\delta}
\newcommand{\phiOne}{\varphi_1}
\newcommand{\Mfun}{\mathcal{M}}
\newcommand{\TIFE}{\textup{TIFE}}
\newcommand{\ETDFigureDir}{figures/learned_potential_etd_experiments}

\begin{document}
\renewcommand{\thefootnote}{\fnsymbol{footnote}}
\title{Data-Driven Free-Energy Learning from Trajectories via Structure-Preserving Exponential Time Differencing}
\author{Wenshuai Hu\thanks{Laboratory of Mathematics and Complex Systems, Ministry of Education and School of Mathematical Sciences, Beijing Normal University, Beijing 100875, China.
   }
    \and
    Guanghua Ji\footnotemark[1]
    \thanks{Corresponding author. Email: \texttt{ghji@bnu.edu.cn}}
    \and Xiao Li\footnotemark[1]
}

\maketitle
\begin{abstract}
Structure-preserving numerical methods for gradient flows have been extensively developed when the governing equations and free energies are known. However, accurately predicting dynamics from observational data while preserving intrinsic physical properties remains challenging. Although neural operators, such as Fourier neural operators (FNOs), provide powerful data-driven approximations, their predictions do not inherently guarantee physical structure preservation. To address this challenge, we integrate deep learning with structure-preserving exponential time differencing (ETD) methods. Instead of directly learning the evolution operator, we introduce a *trajectory-inferred free energy* (TIFE), which reconstructs the unknown energy density from observed trajectories through a Duhamel-based learning framework.  The learned energy determines the variational force and stabilization parameter, enabling structure-preserving numerical evolution. We further extend the framework to jointly identify the free energy and diffusion coefficient. Under suitable assumptions, we establish the maximum-bound principle, unconditional dissipation of the original learned energy, and error estimates incorporating discretization and inference errors. Numerical experiments demonstrate accurate energy recovery and reliable long-time predictions.
\end{abstract}

\begin{keywords}
trajectory-inferred free energy; data-driven gradient flow; exponential time
differencing; maximum-bound principle; energy dissipation; inference
error; diffusion-coefficient identification.
\end{keywords}

\begin{AMS}
65M06, 65M12, 65M15, 35K57.
\end{AMS}

\pagestyle{myheadings}
\thispagestyle{plain}
\markboth{W. Hu, G. Ji, and X. Li}{Trajectory-inferred free energy and ETD}
\section{Introduction}

The starting point of this work is a collection of observed spatiotemporal states rather than a fully specified evolution equation. We allow one or several trajectories,
\begin{equation}
    \mathcal D_{\rm obs}
    =\left\{
      u_{\rm obs}^{(r)}(x,t_n):
      r=1,\ldots,M,\quad n=0,\ldots,N_r
    \right\},
    \label{eq:observational-data}
\end{equation}
which are either generated by an unknown gradient-flow model
or obtained from physical experiments, with boundary conditions, and form of the differential
operator assumed to be known as $\Lop$.   The goal is to infer an effective free-energy density $F_\theta$ from the observed state transitions, and then use it to construct a structure-preserving numerical evolution scheme.
We first treat a supplied operator $\Lop$ and then consider the parameterized Laplacian family $\Lop_a=aA=-a\Delta$, in which the scalar coefficient $a=\varepsilon^2>0$ is also inferred from data.


This distinction between known physical structure and unknown
constitutive information is common in the modeling of dissipative
materials. In phase-field models, physical mechanisms such as
diffusion, interfacial regularization, and elastic interactions
often determine the form of the governing spatial operators,
with a Laplacian appearing in the simplest cases
\cite{cahn1958freeenergy,chen2002phasefield,steinbach2009phasefield}.
However, the associated coefficients and local bulk free-energy
densities may depend on material composition, temperature,
coarse-graining, and unresolved microscopic interactions
\cite{cahn1958freeenergy,flory1942thermodynamics}.
Consequently, prescribing a fixed diffusion coefficient
$a=\varepsilon^2$ and a standard quartic or logarithmic potential,
although convenient for analysis and computation, may not
accurately capture the observed dynamics.
Motivated by recent advances in data-driven identification
of thermodynamic potentials \cite{huang2022vonn},
we investigate whether the effective diffusion coefficient
and free-energy density can be jointly inferred from
state trajectories while retaining the underlying gradient-flow
structure required for reliable, structure-preserving evolution.

Physics-informed neural networks (PINNs) have emerged as a promising approach for solving partial differential equations (PDEs) by incorporating governing equations and initial-boundary conditions into neural network training \cite{raissi2019physics}. For gradient-flow equations, Mattey and Ghosh \cite{mattey2022sequential} developed a sequential training strategy for the Allen--Cahn and Cahn--Hilliard equations to improve long-time predictions. Guo et al. \cite{guo2024adaptive} further introduced an adaptive energy-based training method that exploits the energy-dissipation law to enhance the accuracy and efficiency of PINNs for gradient flows. More recently, K\"ut\"uk and Y\"ucel \cite{kutuk2025energy} incorporated energy-dissipation constraints into PINN training for Allen--Cahn equations. From a complementary perspective, Hu et al. \cite{hu2024energetic} developed energetic variational neural network discretizations with provable energy dissipation. Despite these advances, most such approaches assume that the governing equations or free-energy functionals are specified in advance, whereas recovering unknown constitutive energies directly from state observations remains a distinct challenge.

Another important direction is neural operator learning, which aims to approximate mappings between function spaces rather than individual PDE solutions. Representative architectures include DeepONet \cite{lu2021deeponet} and the Fourier neural operator (FNO) \cite{li2021fno}, which have demonstrated strong capabilities in approximating solution operators for various PDEs. Kovachki et al. \cite{kovachki2023neural} established a general neural operator framework with approximation and discretization-invariance properties. To incorporate physical knowledge, Li et al. \cite{li2024pino} proposed physics-informed neural operators (PINO), combining observational data with PDE residual constraints. Recent studies have also explored energy-aware operator learning. Zhang et al. \cite{zhang2024energy} developed an energy-dissipative evolutionary DeepONet based on the scalar auxiliary variable approach, while Tanaka et al. \cite{tanaka2025energy} introduced energy-consistent neural operators through energy-based training constraints. Nevertheless, learning an accurate evolution operator does not, in general, automatically guarantee the preservation of maximum bounds or the dissipation of the original physical energy during recurrent predictions.

Beyond direct solution and operator approximation, increasing attention has been devoted to learning the underlying thermodynamic potentials from data. Teichert et al. \cite{teichert2019machine} introduced integrable deep neural networks to reconstruct free-energy functions from derivative information, enabling data-driven constitutive modeling in phase-field simulations. Huang et al. \cite{huang2022vonn} developed variational Onsager neural networks (VONNs) to learn free-energy and dissipation potentials from spatiotemporal observations under Onsager's variational principle. They subsequently proposed statistical-physics-informed neural networks to address the nonuniqueness of thermodynamic model identification from macroscopic data \cite{huang2025statpinn}. More recently, Cohen et al. \cite{cohen2026differentiable} developed a differentiable framework for identifying free energies and kinetic laws from noisy spatiotemporal observations of pattern-forming systems. These studies demonstrate the potential of learning constitutive energy structures rather than directly approximating solution trajectories. However, the recovery of an energy functional alone does not necessarily provide a numerical evolution scheme with rigorous discrete structure-preserving properties.

Taken together, these developments highlight the growing importance of integrating machine learning with the physical structure of dissipative systems. Nevertheless, relatively few studies have simultaneously addressed the identification of unknown free energies from state-only trajectories and the construction of numerical updates that rigorously preserve the intrinsic structures of the inferred dynamics. In particular, a unified framework connecting trajectory-based energy learning, physical-parameter identification, unconditional dissipation of the original learned energy, maximum-bound preservation, and rigorous error analysis remains comparatively underdeveloped. Motivated by this gap, we combine deep learning with exponential time differencing (ETD) through a trajectory-inferred free-energy framework. Rather than directly learning the evolution operator, we identify the unknown energy and physical coefficient through the ETD response structure and subsequently construct a structure-preserving numerical evolution driven by the learned variational force.

Motivated by these considerations, we develop a data-driven,
structure-preserving framework that integrates deep learning
with exponential time differencing (ETD) methods. The central
idea is to recover the underlying constitutive energy from
observed state trajectories rather than directly learning
the full-field evolution operator. We introduce the concept
of \emph{trajectory-inferred free energy} (\TIFE), represented
by a smooth function $F_\theta$, which is identified from
state observations without requiring energy or nonlinear-force
labels. In addition to the unknown free-energy density,
the framework also allows the identification of an unknown
positive diffusion coefficient, while retaining the known
gradient-flow structure and the form of the differential
operator.

By contrast, the proposed method does not directly learn
$\Phi_{\Delta t,h}$ as an unconstrained field map.
Instead, it identifies the constitutive free-energy density
$F_\theta$ and, when unknown, the positive diffusion
coefficient $a_\alpha$, from which a structure-preserving
ETD update is constructed:
\begin{align}
    (a_\alpha,F_\theta)
    &\longrightarrow
    \bigl(\Lop_h(a_\alpha),N_\theta=-F_\theta'\bigr)
    \xrightarrow{\text{ETD}}
    \Phi_{\alpha,\theta,\tau,h}^{\TIFE\text{-ETD}},
    \label{eq:energy-to-update-factorization}\\
    \Phi_{\alpha,\theta,\tau,h}^{\TIFE\text{-ETD}}(v_h)
    &=
    e^{-\tau\Lop_{\alpha,\kappa,h}}v_h
    +\tau\phiOne(\tau\Lop_{\alpha,\kappa,h})
    \bigl[\kappa v_h-F_\theta'(v_h)\bigr],
    \label{eq:intro-learned-etd-map}
\end{align}
where $\Lop_{\alpha,\kappa,h}=a_\alpha A_h+\kappa I$.
When the linear coefficient is known, $a_\alpha A_h$
is replaced by the prescribed operator $\Lop_h$.

The proposed framework consists of two closely connected
stages: data-driven energy identification and
structure-preserving numerical evolution.
In the first stage, we exploit the variation-of-constants
formula to incorporate the ETD response structure into
the learning process. By matching the predicted and
observed state transitions, we infer the unknown
free-energy density and, when necessary, the diffusion
coefficient, without requiring energy or nonlinear-force
labels. In the second stage, the learned free energy
determines the nonlinear variational force, while its
curvature guides the selection of the stabilization
parameter. These quantities, together with the known
differential-operator structure, are incorporated into
the ETD scheme to produce physically consistent
numerical predictions.

The overall procedure is summarized as
\[
\boxed{
\begin{gathered}
\text{Observed state trajectories}\\
\Downarrow\\
\text{ETD-response-based energy learning}\\
\Downarrow\\
\text{Inferred free energy and diffusion coefficient}\\
\Downarrow\\
\text{Curvature-stabilized \TIFE--ETD evolution}.
\end{gathered}
}
\]

Unlike direct evolution-operator learning, this
construction retains the underlying gradient-flow
structure throughout both identification and evolution.
Under suitable assumptions, the resulting scheme
preserves the maximum-bound principle and unconditionally
dissipates the original learned energy without introducing
auxiliary energy variables. We further establish
conditional identification estimates and fully discrete
error bounds that quantify the effects of temporal and
spatial discretization, free-energy inference, and
diffusion-coefficient identification.
The detailed construction and theoretical analysis
are presented in the subsequent sections.

\paragraph{Main contributions.}
The main contributions of this work are threefold.
First, we propose a trajectory-based free-energy learning
framework that exploits the ETD response structure to
identify unknown constitutive energies directly from
observed state trajectories, without requiring energy or
nonlinear-force labels. The framework also accommodates
the joint identification of an unknown diffusion coefficient.
Second, we integrate the inferred free energy with a
stabilized ETD scheme, where the learned variational force
governs the numerical evolution and the energy curvature
guides the stabilization. Under suitable assumptions,
the resulting method preserves the maximum-bound principle
and unconditionally dissipates the original learned energy,
without introducing auxiliary energy variables.
Third, we establish conditional data-to-force and joint
identifiability estimates, together with fully discrete
error bounds that distinguish temporal and spatial
discretization errors from energy and coefficient
inference errors. Numerical experiments further demonstrate
the effectiveness of the proposed framework in recovering
constitutive information and producing accurate,
structure-preserving long-time predictions.

The main contributions of this work are summarized as follows:
\begin{itemize}
    \item We propose a data-driven free-energy learning framework
    based on the ETD response structure, enabling the identification
    of unknown free energies and diffusion coefficients directly
    from state trajectories without energy or force labels.

    \item We develop a stabilized ETD scheme driven by the learned
    free energy, which preserves the maximum-bound principle and
    unconditionally dissipates the original learned energy
    under suitable assumptions.

    \item We establish conditional identifiability estimates and
    fully discrete error bounds that quantify the effects of
    discretization and inference errors. Numerical experiments
    demonstrate accurate energy recovery and reliable long-time
    predictions.
\end{itemize}

\paragraph{Organization of the paper.}
The remainder of this paper is organized as follows. Section~2 introduces the
spatial discretization, the \TIFE{} parameterization, and the state-pair
training objective; it also establishes a quantitative response-to-force
relation and formulates the joint identification of the local free energy and
the scalar diffusion coefficient. Section~3 presents the curvature-stabilized
\TIFE-based ETD1 scheme. Section~4 establishes the maximum-bound principle and
the unconditional dissipation of the original \TIFE{} functional. Section~5
derives fully discrete error estimates that distinguish temporal and spatial
discretization errors from force- and coefficient-inference errors. Section~6
reports numerical experiments on potential recovery, long-time prediction,
transfer across discretizations and physical parameters, structural
preservation, and robustness with respect to data volume and observation
noise. Finally, Section~7 summarizes the main conclusions and discusses
possible extensions.

\section{Inferring the \TIFE{} from trajectory data}
\subsection{Spatial discretization}
Let $X_h$ denote a grid-function space equipped with the discrete inner product $\innerh{\cdot}{\cdot}$ and norm $\normh{\cdot}$. For a linear map $T_h:X_h\to X_h$, let
\begin{equation}
    \norm{T_h}_{h\to h}
    :=\sup_{0\ne v_h\in X_h}\frac{\normh{T_h v_h}}{\normh{v_h}}
    \label{eq:discrete-operator-norm}
\end{equation}
denote the induced operator norm. Let $\Lop_h:X_h\to X_h$ be a symmetric nonnegative approximation of $\Lop$. We define the discrete \TIFE{} functional by
\begin{equation}
    \E_{\theta,h}[v_h]
    :=\frac12\innerh{v_h}{\Lop_h v_h}
      +\innerh{F_\theta(v_h)}{1}.
    \label{eq:tife-discrete-energy}
\end{equation}
For the pointwise nonlinear estimates below, we assume that the discrete inner product is induced by a positive-weight nodal quadrature,
\begin{equation}
    \innerh{v_h}{w_h}
    =\sum_{j\in\mathcal J_h}\omega_j v_j w_j,
    \qquad \omega_j>0,
    \label{eq:positive-quadrature}
\end{equation}
and we write $\norm{v_h}_{\infty,h}:=\max_{j\in\mathcal J_h}|v_j|$. In particular, a pointwise inequality may be multiplied by $\omega_j$ and summed, and $\normh{1}^2=\sum_j\omega_j$ is assumed to be bounded independently of $h$ (it equals $|\Omega|$ for the standard periodic quadrature).
The function $F_\theta$ is applied pointwise to $v_h$.

For the maximum-bound result we impose an additional semigroup assumption.

\begin{assumption}[Positivity-preserving discrete semigroup]
\label{ass:positive-semigroup}
The discrete operator $\Lop_h$ generates a positivity-preserving $\ell_h^\infty$-contractive semigroup:
\begin{equation}
    \norm{e^{-t\Lop_h}v_h}_{\infty,h}
    \le \norm{v_h}_{\infty,h},
    \qquad t\ge0.
    \label{eq:semigroup-contractivity}
\end{equation}
\end{assumption}

For example, this assumption is satisfied by standard second-order finite-difference discretizations of the Laplacian under periodic or suitable Neumann boundary conditions. The energy-dissipation theorem itself does not require the maximum principle; it only requires that the numerical states remain in a set on which the \TIFE{} density has a bounded upper curvature.
\subsection{\TIFE{} parameterization}
Assume that the linear operator $\Lop$ is known, while the nonlinear free-energy density $F$ is unknown. We parameterize a smooth scalar function
\begin{equation}
    F_\theta:\R\to\R
\end{equation}
by a differentiable model, for example a smooth multilayer perceptron, a spline, a polynomial basis, or another $C^2$ representation. The nonlinear force is not inferred independently. Instead, it is defined by
\begin{equation}
    N_\theta(z):=-F_\theta'(z).
    \label{eq:tife-force}
\end{equation}
Hence the continuous model induced by the \TIFE{} is automatically a gradient flow,
\begin{equation}
    u_t+\Lop u=-F_\theta'(u)
    =-\frac{\delta\E_\theta}{\delta u},
    \label{eq:tife-gradient-flow}
\end{equation}
with
\begin{equation}
    \E_\theta[u]
    =\frac12\inner{u}{\Lop u}
     +\int_\Omega F_\theta(u)\,dx.
    \label{eq:tife-energy}
\end{equation}
The additive constant in $F_\theta$ is immaterial for the dynamics; a convenient gauge such as $F_\theta(0)=0$ may therefore be imposed.

\subsection{State-pair training objective}
Let
\[
    \{(U^m,U^{m+1})\}_{m=1}^{N_{\rm data}}
\]
be state pairs separated by a data sampling interval $\Delta t$. The index $m$ enumerates the available state pairs after pooling one or several observed trajectories; it does not impose a single-trajectory assumption. Define
\begin{equation}
    S_{\Delta t,h}:=e^{-\Delta t\Lop_h},
    \qquad
    B_{\Delta t,h}:=\Delta t\,\phiOne(\Delta t\Lop_h),
\end{equation}
where
\begin{equation}
    \phiOne(z)=\frac{1-e^{-z}}{z},
    \qquad \phiOne(0)=1.
\end{equation}
A one-step model prediction is
\begin{equation}
    \widehat U_\theta^{m+1}
    =S_{\Delta t,h}U^m
     -B_{\Delta t,h}F_\theta'(U^m).
    \label{eq:training-predictor}
\end{equation}
Equivalently, one may infer the Duhamel contribution
\begin{equation}
    D^m:=D_{\Delta t,{\rm data}}^m
    :=U^{m+1}-S_{\Delta t,h}U^m,
    \qquad
    D_\theta^m:=-B_{\Delta t,h}F_\theta'(U^m).
\end{equation}
A normalized training loss is then
\begin{equation}
    \mathcal J_D(\theta)
    =\frac{\sum_m\normh{D_\theta^m-D^m}^2}
    {\sum_m\normh{D^m}^2+\varepsilon_D},
    \label{eq:duhamel-loss}
\end{equation}
where $\varepsilon_D>0$ is a denominator safeguard with the units of the
summed squared response norm.
possibly augmented by a small state loss
\begin{equation}
    \mathcal J_{\rm state}(\theta)
    =\frac{1}{N_{\rm data}}\sum_m
    \normh{\widehat U_\theta^{m+1}-U^{m+1}}^2.
\end{equation}
Because
\[
    \widehat U_\theta^{m+1}-U^{m+1}=D_\theta^m-D^m,
\]
the state residual and the Duhamel residual contain the same information; a
separately normalized state term only reweights the same residual. The
training data contain neither energy labels nor nonlinear-force labels. The
\TIFE{} density is identified through the ETD1 response factorization of the
observed state evolution in~\eqref{eq:training-predictor}.

\subsection{From response-space training error to force error}
\label{sec:response-to-force}

We next make explicit when the response loss controls the instantaneous
variational force.  This step is necessarily conditional because
$B_{\Delta t,h}$ is a smoothing operator and because finite-step trajectory
data contain an ETD1 freezing defect.

Let $\overline U_h(t)$ be a noise-free solution of the semidiscrete true
model
\begin{equation}
    \frac{d}{dt}\overline U_h(t)+\Lop_h\overline U_h(t)
    =N(\overline U_h(t)),
    \qquad N(z):=-F'(z),
    \label{eq:clean-semidiscrete-data}
\end{equation}
and write
\begin{equation}
    U^m=\overline U_h(t_m)+\eta_-^m,
    \qquad
    U^{m+1}=\overline U_h(t_m+\Delta t)+\eta_+^m,
    \label{eq:observation-perturbations}
\end{equation}
for the observed input and output states.  Model discrepancy and errors in
the known linear response would enter the estimates in the same additive way
as an additional response perturbation.  Assume that the
clean and observed states take values in a compact interval
$\Aset_{\rm data}$ and set
\begin{equation}
    L_N:=\norm{F''}_{L^\infty(\Aset_{\rm data})},
    \qquad
    C_t:=\sup_t\normh{\partial_t\overline U_h(t)}.
    \label{eq:data-regularity-constants}
\end{equation}
The exact Duhamel contribution admits the decomposition
\begin{align}
    &\overline U_h(t_m+\Delta t)
      -S_{\Delta t,h}\overline U_h(t_m)
      \nonumber\\
    &\qquad
      =B_{\Delta t,h}N(\overline U_h(t_m))+q^m,
    \label{eq:duhamel-freezing-decomposition}
\end{align}
where semigroup contractivity gives
\begin{equation}
    \normh{q^m}
    \le \frac12 L_N C_t\Delta t^2.
    \label{eq:duhamel-freezing-defect}
\end{equation}

To quantify response inversion, let
\begin{equation}
    \Lambda_h:=\lambda_{\max}(\Lop_h),
    \qquad
    \beta_{\Delta t,h}
    :=\Delta t\,\phiOne(\Delta t\Lambda_h)>0.
    \label{eq:response-coercivity-constant}
\end{equation}
Since $\Lop_h$ is symmetric nonnegative, spectral calculus yields
\begin{equation}
    \beta_{\Delta t,h}\normh{v_h}
    \le \normh{B_{\Delta t,h}v_h}
    \le \Delta t\,\normh{v_h},
    \qquad v_h\in X_h.
    \label{eq:response-coercivity}
\end{equation}
Thus $\beta_{\Delta t,h}^{-1}$ is the explicit price of recovering an
instantaneous force from its propagated response.

Define the response residual and its root-mean-square value by
\begin{equation}
    r_\theta^m:=D_\theta^m-D^m,
    \qquad
    \mathcal R_D(\theta)
    :=\left(\frac1{N_{\rm data}}
      \sum_m\normh{r_\theta^m}^2\right)^{1/2}.
    \label{eq:response-rms-residual}
\end{equation}
The normalized training loss determines this quantity exactly:
\begin{equation}
    \mathcal R_D(\theta)
    =\left(
      \frac{\sum_m\normh{D^m}^2+\varepsilon_D}{N_{\rm data}}
      \mathcal J_D(\theta)
      \right)^{1/2}.
    \label{eq:normalized-loss-to-residual}
\end{equation}
We also introduce
\begin{align}
    \mathfrak N_{{\rm obs},2}
    &:={}
    \left[\frac1{N_{\rm data}}\sum_m
    \left(
      \normh{\eta_+^m}
      +(1+\Delta t L_N)\normh{\eta_-^m}
    \right)^2\right]^{1/2},
    \label{eq:observation-rms-error}\\
    \delta_{F,{\rm emp}}
    &:={}
    \left[\frac1{N_{\rm data}}\sum_m
    \normh{F_\theta'(U^m)-F'(U^m)}^2\right]^{1/2}.
    \label{eq:empirical-force-error}
\end{align}

\begin{theorem}[Response residual controls empirical force error]
\label{thm:response-to-empirical-force}
Suppose that~\eqref{eq:clean-semidiscrete-data}--\eqref{eq:data-regularity-constants}
hold and that $F,F_\theta\in C^2(\Aset_{\rm data})$. Then
\begin{equation}
    \boxed{
    \delta_{F,{\rm emp}}
    \le
    \beta_{\Delta t,h}^{-1}
    \left(
       \mathcal R_D(\theta)
       +\mathfrak N_{{\rm obs},2}
       +\frac12 L_N C_t\Delta t^2
    \right).}
    \label{eq:empirical-force-from-training}
\end{equation}
Consequently, the training loss controls the empirical force error
through~\eqref{eq:normalized-loss-to-residual}, with explicit contributions from
observation error and the finite-step freezing defect.
\end{theorem}

\begin{proof}
Variation of constants in~\eqref{eq:clean-semidiscrete-data}
gives~\eqref{eq:duhamel-freezing-decomposition}, with
\[
    q^m=\int_0^{\Delta t}
    e^{-(\Delta t-s)\Lop_h}
    \bigl[N(\overline U_h(t_m+s))-N(\overline U_h(t_m))\bigr],ds.
\]
The Lipschitz continuity of $N$, semigroup contractivity, and
$\normh{\overline U_h(t_m+s)-\overline U_h(t_m)}\le C_t s$
imply~\eqref{eq:duhamel-freezing-defect}. Moreover,
\begin{align*}
 &B_{\Delta t,h}
 \bigl[N_\theta(U^m)-N(U^m)\bigr]\\
 &\quad={}
 r_\theta^m+q^m+\eta_+^m-S_{\Delta t,h}\eta_-^m
 +B_{\Delta t,h}
 \bigl[N(\overline U_h(t_m))-N(U^m)\bigr].
\end{align*}
Using $\norm{S_{\Delta t,h}}_{h\to h}\le1$,
$\norm{B_{\Delta t,h}}_{h\to h}\le\Delta t$, and the Lipschitz constant
$L_N$ gives
\begin{align*}
 \normh{B_{\Delta t,h}[N_\theta(U^m)-N(U^m)]}
 \le{}&\normh{r_\theta^m}+\frac12 L_N C_t\Delta t^2
       +\normh{\eta_+^m}\\
 &+(1+\Delta t L_N)\normh{\eta_-^m}.
\end{align*}
The lower bound in~\eqref{eq:response-coercivity}, followed by the
Minkowski inequality in the data index,
proves~\eqref{eq:empirical-force-from-training} because
$N_\theta-N=-(F_\theta'-F')$.
\end{proof}

The empirical estimate does not by itself control unseen state values.  To
state a deterministic uniform extension, let
\begin{equation}
    \mathcal Z_h:=\{U_j^m:\ 1\le m\le N_{\rm data},\ j\in\mathcal J_h\},
    \qquad
    \rho_{\rm data}:=\sup_{z\in\Aset_{\rm data}}
      \min_{\zeta\in\mathcal Z_h}|z-\zeta|,
    \label{eq:state-fill-distance}
\end{equation}
and let $\omega_{\min,h}:=\min_{j\in\mathcal J_h}\omega_j$.  Further set
\begin{align}
    \mathcal R_{D,\infty}(\theta)
    &:=\max_m\normh{r_\theta^m},
    \label{eq:max-response-residual}\\
    \mathfrak N_{{\rm obs},\infty}
    &:=\max_m\left\{
      \normh{\eta_+^m}
      +(1+\Delta t L_N)\normh{\eta_-^m}
      \right\},
    \label{eq:max-observation-error}\\
    L_g&:=\norm{F_\theta''-F''}_{L^\infty(\Aset_{\rm data})}.
    \label{eq:force-error-lipschitz}
\end{align}
In particular,
$L_g\le\norm{F_\theta''}_{L^\infty(\Aset_{\rm data})}+L_N$, so the
coverage term requires only an a priori true-force Lipschitz bound together
with the computable curvature of the learned potential.

\begin{corollary}[Uniform force error on the data-supported interval]
\label{cor:uniform-force-from-training}
Under the assumptions of Theorem~\ref{thm:response-to-empirical-force},
\begin{equation}
    \boxed{
    \norm{F_\theta'-F'}_{L^\infty(\Aset_{\rm data})}
    \le \mathfrak E_{\rm data},}
    \label{eq:uniform-force-training-bound}
\end{equation}
where
\begin{equation}
    \mathfrak E_{\rm data}
    :=\frac{
       \mathcal R_{D,\infty}(\theta)
       +\mathfrak N_{{\rm obs},\infty}
       +\frac12 L_N C_t\Delta t^2
    }{\sqrt{\omega_{\min,h}}\,\beta_{\Delta t,h}}
    +L_g\rho_{\rm data}.
    \label{eq:data-force-estimator}
\end{equation}
In particular, the residual term is controlled by the optimized loss since
\begin{equation}
    \mathcal R_{D,\infty}(\theta)
    \le
    \left[
      \left(\sum_m\normh{D^m}^2+\varepsilon_D\right)
      \mathcal J_D(\theta)
    \right]^{1/2}.
    \label{eq:max-residual-from-loss}
\end{equation}
\end{corollary}

\begin{proof}
For every observed nodal value $U_j^m$, the argument in
Theorem~\ref{thm:response-to-empirical-force}, together
with~\eqref{eq:response-coercivity} and the positive quadrature weights, gives
\begin{align*}
 |F_\theta'(U_j^m)-F'(U_j^m)|
 \le{}&\omega_{\min,h}^{-1/2}
 \normh{F_\theta'(U^m)-F'(U^m)}\\
 \le{}&
 \frac{\normh{r_\theta^m}
       +\mathfrak N_{{\rm obs},\infty}
       +\frac12 L_N C_t\Delta t^2}
      {\sqrt{\omega_{\min,h}}\,\beta_{\Delta t,h}}.
\end{align*}
For any $z\in\Aset_{\rm data}$, choose $\zeta\in\mathcal Z_h$ with
$|z-\zeta|\le\rho_{\rm data}$.  The scalar function
$F_\theta'-F'$ is $L_g$-Lipschitz, so the preceding nodal estimate plus
$L_g\rho_{\rm data}$ proves~\eqref{eq:data-force-estimator}.
Finally,~\eqref{eq:max-residual-from-loss} follows from
$\max_m\normh{r_\theta^m}\le
(\sum_m\normh{r_\theta^m}^2)^{1/2}$ and~\eqref{eq:duhamel-loss}.
\end{proof}

\begin{remark}[Conditioning and scope]
The factors in~\eqref{eq:data-force-estimator} quantify three distinct
limitations. For a
second-order elliptic discretization, $\Lambda_h=O(h^{-2})$, so inversion of
the smoothed response can amplify unresolved high-frequency errors.  The
factor $\omega_{\min,h}^{-1/2}$ is the conservative cost of converting a
discrete $L^2$ residual into a nodal maximum, while $\rho_{\rm data}$ excludes
unsupported extrapolation.  Consequently, a mesh-uniform uniform-force
estimate requires additional sampling stability or regularization; a small
mean-square training loss alone is insufficient.  For clean, sufficiently
dense data at fixed $h$,~\eqref{eq:data-force-estimator} reduces to a
quantitative response-residual and $O(\Delta t)$ force-bias estimate because
$\beta_{\Delta t,h}\sim\Delta t$ as $\Delta t\Lambda_h\to0$.
\end{remark}

\begin{remark}
The analysis does not depend on the specific inference architecture or
optimization algorithm.  The bridge estimate uses only the achieved response
residual, observation accuracy, state coverage, and smoothness of the true and
inferred forces.  The subsequent structure-preserving analysis additionally
uses the endpoint behavior and the upper curvature of $F_\theta$.
\end{remark}

\subsection{Joint inference of a scalar Laplacian coefficient and the \TIFE}
\label{sec:joint-coefficient-energy}

We now retain the known Laplacian structure but allow its scalar coefficient
to be inferred jointly with the local energy.  Let
\begin{equation}
    A_h:=-\Delta_h,
    \qquad
    \Lop_h(a):=aA_h,
    \qquad
    a_\star=\varepsilon_\star^2>0,
    \label{eq:parameterized-laplacian}
\end{equation}
where $a_\star$ denotes the true coefficient. Fix a physically admissible
interval
\begin{equation}
    \mathcal I_a:=[a_{\min},a_{\max}],
    \qquad 0\le a_{\min}<a_{\max},
    \qquad a_\star\in(a_{\min},a_{\max}),
    \label{eq:admissible-coefficient-interval}
\end{equation}
and use, for example, the bounded parameterization
\begin{equation}
    a_\alpha
    =a_{\min}+(a_{\max}-a_{\min})\sigma(\alpha),
    \qquad
    \sigma(\alpha):=\frac{1}{1+e^{-\alpha}}.
    \label{eq:positive-coefficient-parameterization}
\end{equation}
For any candidate $a\in\mathcal I_a$, define
\begin{equation}
    S_{\Delta t,h}(a):=e^{-\Delta t\Lop_h(a)},
    \qquad
    B_{\Delta t,h}(a)
    :=\Delta t\,\phiOne(\Delta t\Lop_h(a)).
    \label{eq:coefficient-dependent-response}
\end{equation}
The coefficient and energy are inferred jointly through the ETD predictor
\begin{equation}
    \widehat U_{\alpha,\theta}^{m+1}
    =S_{\Delta t,h}(a_\alpha)U^m
     +B_{\Delta t,h}(a_\alpha)N_\theta(U^m),
    \qquad N_\theta=-F_\theta',
    \label{eq:joint-etd-predictor}
\end{equation}
and the normalized loss
\begin{equation}
    \mathcal J_{a,F}(\alpha,\theta)
    :=\frac{
      \sum_m\normh{
      \widehat U_{\alpha,\theta}^{m+1}-U^{m+1}}^2}
      {\sum_m\normh{U^{m+1}}^2+\varepsilon_U}.
    \label{eq:joint-coefficient-energy-loss}
\end{equation}
Here $\varepsilon_U>0$ is a denominator safeguard with the units of the
summed squared state norm.
The corresponding root-mean-square observed residual is
\begin{equation}
    \mathcal R_{a,F}^{\rm obs}(\alpha,\theta)
    :=\left[\frac1{N_{\rm data}}\sum_m
      \normh{\widehat U_{\alpha,\theta}^{m+1}-U^{m+1}}^2
    \right]^{1/2}
    =\left[
      \frac{\sum_m\normh{U^{m+1}}^2+\varepsilon_U}{N_{\rm data}}
      \mathcal J_{a,F}(\alpha,\theta)
    \right]^{1/2}.
    \label{eq:joint-observed-residual}
\end{equation}
Thus the inferred coefficient affects both the linear response and the
propagation of the nonlinear force; it is not fitted in a separate regression
step.

Joint recovery is not automatic.  To state the required data-richness
condition, introduce the frozen-force ETD map
\begin{equation}
    \Phi_{\Delta t,h}(a,G;V)
    :=S_{\Delta t,h}(a)V+B_{\Delta t,h}(a)G(V).
    \label{eq:joint-frozen-map}
\end{equation}
Let $\overline U^m:=\overline U_h(t_m)$ denote the noise-free input states
generated with $(a_\star,N)$, where $N=-F'$. To prevent an unrestricted local
force from absorbing arbitrary changes in the coefficient on finitely many
samples, let $\mathfrak F\subset C^2(\Aset_{\rm data})$ be an admissible energy
class containing the true density $F$ and all candidate densities
$F_\theta$ under consideration, and define
\begin{equation}
    \mathfrak G
    :=\{-\widetilde F':\ \widetilde F\in\mathfrak F\}.
    \label{eq:admissible-force-class}
\end{equation}
Assume that every $G\in\mathfrak G$ is Lipschitz continuous on
$\Aset_{\rm data}$ with a common constant $L_{\mathfrak G}$. The class may be
specified through the architecture, regularization, or explicit smoothness and
complexity constraints. Finally, let $c_a>0$ be a fixed scaling constant that
nondimensionalizes the coefficient error.

\begin{assumption}[Joint response identifiability]
\label{ass:joint-identifiability}
There exists $C_{\rm id}>0$ such that, for every
$(a,G)\in\mathcal I_a\times\mathfrak G$,
\begin{align}
 &\left[
 c_a^2|a-a_\star|^2
 +\frac1{N_{\rm data}}\sum_m
   \normh{G(\overline U^m)-N(\overline U^m)}^2
 \right]^{1/2}
 \nonumber\\
 &\qquad\le
 C_{\rm id}
 \left[
 \frac1{N_{\rm data}}\sum_m
 \normh{
 \Phi_{\Delta t,h}(a,G;\overline U^m)
 -\Phi_{\Delta t,h}(a_\star,N;\overline U^m)}^2
 \right]^{1/2}.
 \label{eq:joint-identifiability}
\end{align}
\end{assumption}

\begin{theorem}[Joint response error controls coefficient and force errors]
\label{thm:joint-response-error}
Suppose that Assumption~\ref{ass:joint-identifiability} holds,
$(a_\alpha,N_\theta)\in\mathcal I_a\times\mathfrak G$, the clean semidiscrete
trajectory remains in $\Aset_{\rm data}$, and $L_N$ and $C_t$ are defined as
in~\eqref{eq:data-regularity-constants}. Set
\begin{equation}
    \mathcal R_{a,F}^{\rm clean}(\alpha,\theta)
    :=\left[
    \frac1{N_{\rm data}}\sum_m
    \normh{
    \Phi_{\Delta t,h}(a_\alpha,N_\theta;\overline U^m)
    -\overline U^{m+1}}^2
    \right]^{1/2}.
    \label{eq:joint-clean-residual}
\end{equation}
Then
\begin{align}
 &\left[
 c_a^2|a_\alpha-a_\star|^2
 +\frac1{N_{\rm data}}\sum_m
 \normh{F_\theta'(\overline U^m)-F'(\overline U^m)}^2
 \right]^{1/2}
 \nonumber\\
 &\qquad\le
 C_{\rm id}\left(
 \mathcal R_{a,F}^{\rm clean}(\alpha,\theta)
 +\frac12 L_N C_t\Delta t^2
 \right).
 \label{eq:joint-training-to-parameters}
\end{align}
\end{theorem}

\begin{proof}
Variation of constants for the true semidiscrete equation gives
\begin{equation*}
    \overline U^{m+1}
    =\Phi_{\Delta t,h}
      (a_\star,N;\overline U^m)+q^m,
    \qquad
    \normh{q^m}\le\frac12 L_N C_t\Delta t^2,
\end{equation*}
exactly as in~\eqref{eq:duhamel-freezing-defect}. The triangle inequality in
the data index bounds the response difference on the right-hand side
of~\eqref{eq:joint-identifiability} by
$\mathcal R_{a,F}^{\rm clean}+\frac12 L_N C_t\Delta t^2$. Taking
$a=a_\alpha$ and $G=N_\theta$, and using
$N_\theta-N=-(F_\theta'-F')$,
proves~\eqref{eq:joint-training-to-parameters}.
\end{proof}

For the perturbed observations in~\eqref{eq:observation-perturbations}, the
uniform force-class Lipschitz bound gives
\begin{equation}
    \mathcal R_{a,F}^{\rm clean}(\alpha,\theta)
    \le \mathcal R_{a,F}^{\rm obs}(\alpha,\theta)
    +\mathfrak N_{{\rm joint},2},
    \label{eq:joint-clean-to-observed}
\end{equation}
where
\begin{equation}
    \mathfrak N_{{\rm joint},2}
    :=\left[\frac1{N_{\rm data}}\sum_m
      \left(
        (1+\Delta t L_{\mathfrak G})\normh{\eta_-^m}
        +\normh{\eta_+^m}
      \right)^2
    \right]^{1/2}.
    \label{eq:joint-observation-error}
\end{equation}
Indeed, $\Phi_{\Delta t,h}(a,G;\cdot)$ has Lipschitz constant at most
$1+\Delta t L_{\mathfrak G}$ in $\normh{\cdot}$. Consequently,
Theorem~\ref{thm:joint-response-error} also holds with
$\mathcal R_{a,F}^{\rm clean}(\alpha,\theta)$ replaced by
$\mathcal R_{a,F}^{\rm obs}(\alpha,\theta)+\mathfrak N_{{\rm joint},2}$.

\begin{remark}[Meaning of joint identifiability]
Assumption~\ref{ass:joint-identifiability} requires spatially rich data. The
coefficient multiplies $A_h U$, whereas the local force depends only on the
state value. Observations in which comparable state values occur with
different spatial curvatures help separate the two effects. If instead
$A_h U^m$ is approximately a single-valued function of $U^m$ throughout the
data, a change in $a$ can be absorbed into the inferred local force and the
assumption may fail. Multiple initial conditions and a mixture of low- and
high-frequency modes are therefore more important than the raw number of
nearly redundant snapshots. The condition is a restricted stability
assumption on $\mathcal I_a\times\mathfrak G$, not a generic identifiability
statement for arbitrary functions. The time scale and mobility are kept
fixed and known; if an unknown mobility is introduced simultaneously, only
products of the mobility with $a$ and $F'$ are generally identifiable without
an additional normalization.
\end{remark}

For an inferred pair $(a_\alpha,F_\theta)$, the corresponding discrete energy is
\begin{equation}
    \E_{\alpha,\theta,h}[v_h]
    :=\frac12\innerh{v_h}{\Lop_h(a_\alpha)v_h}
      +\innerh{F_\theta(v_h)}{1}.
    \label{eq:joint-learned-energy}
\end{equation}
If $A_h=-\Delta_h$ satisfies the usual summation-by-parts identity, define the
discrete gradient seminorm by
\begin{equation}
    \norm{\nabla_h v_h}_{2,h}^2:=\innerh{v_h}{A_h v_h}.
    \label{eq:discrete-gradient-seminorm}
\end{equation}
Then
\begin{equation}
    \E_{\alpha,\theta,h}[v_h]
    =\frac{a_\alpha}{2}\norm{\nabla_h v_h}_{2,h}^2
      +\innerh{F_\theta(v_h)}{1}.
    \label{eq:joint-learned-energy-gradient}
\end{equation}

\section{Curvature-informed \TIFE-based ETD scheme}

\subsection{Stabilized splitting}
Let $\Aset\subset\R$ be a convex admissible interval. In the maximum-bound setting we take
\begin{equation}
    \Aset=[-\beta,\beta],
    \qquad \beta>0.
\end{equation}
Given the \TIFE{} density, define
\begin{equation}
    L_{\theta,\Aset}^+
    :=\sup_{z\in\Aset}F_\theta''(z).
    \label{eq:curvature-bound}
\end{equation}
Choose
\begin{equation}
    \kappa\ge L_{\theta,\Aset}^+.
    \label{eq:kappa-condition}
\end{equation}
Throughout the analysis we make the nonnegative choice
\begin{equation}
    \kappa\ge \max\{0,L_{\theta,\Aset}^+\}.
    \label{eq:nonnegative-kappa}
\end{equation}
This supplementary condition guarantees that $\Lop_{\kappa,h}=\Lop_h+\kappa I$ is nonnegative and that $e^{-t\Lop_{\kappa,h}}=e^{-\kappa t}e^{-t\Lop_h}$ is contractive. It is also needed in the maximum-bound argument, where $\kappa\beta$ is used as a nonnegative bound.
The parameter $\kappa$ is fixed independently of the discretization parameters $h$ and $\tau$; for a family of \TIFE{} models, a uniform error constant requires a uniform admissible bound for $\kappa$.
In computations one may use a dense sampling or a certified upper bound and add a small safety margin.
A value obtained only by dense sampling is an empirical candidate for $\kappa$; the theorems require an actual upper bound on the whole admissible interval. Such a bound may, for example, be certified by interval arithmetic, an analytic estimate for the chosen representation, or a verified global maximization procedure.

Introduce
\begin{equation}
    \Lop_{\kappa,h}:=\Lop_h+\kappa I,
    \qquad
    \Lop_\kappa:=\Lop+\kappa I,
    \qquad
    H_{\kappa,\theta}(z):=\kappa z-F_\theta'(z).
    \label{eq:stabilized-objects}
\end{equation}
The \TIFE-induced gradient flow is unchanged, since
\begin{equation}
    u_t+\Lop_{\kappa}u=H_{\kappa,\theta}(u).
\end{equation}

\subsection{\TIFE-based ETD1 scheme}
The proposed stabilized \TIFE-based ETD1 scheme is
\begin{equation}
    \boxed{
    u_h^{n+1}
    =e^{-\tau\Lop_{\kappa,h}}u_h^n
    +\tau\phiOne(\tau\Lop_{\kappa,h})
      H_{\kappa,\theta}(u_h^n).
    }
    \label{eq:etd1}
\end{equation}
For the jointly inferred Laplacian coefficient, the same update is used with
\begin{equation}
    \Lop_h=\Lop_h(a_\alpha)=a_\alpha A_h,
    \qquad
    \Lop_{\alpha,\kappa,h}:=a_\alpha A_h+\kappa I,
    \label{eq:joint-deployment-operator}
\end{equation}
so that
\begin{equation}
    u_h^{n+1}
    =e^{-\tau\Lop_{\alpha,\kappa,h}}u_h^n
    +\tau\phiOne(\tau\Lop_{\alpha,\kappa,h})
      \bigl[\kappa u_h^n-F_\theta'(u_h^n)\bigr].
    \label{eq:joint-coefficient-etd1}
\end{equation}
For later use, define
\begin{equation}
    \Mfun(z):=\frac{z}{e^z-1},
    \qquad \Mfun(0)=1.
    \label{eq:M-function}
\end{equation}

\begin{lemma}[Coercive reformulation]
\label{lem:coercive-reformulation}
Scheme~\eqref{eq:etd1} is equivalent to
\begin{equation}
    \Mfun(\tau\Lop_{\kappa,h})\dt u_h^{n+1}
    +\Lop_{\kappa,h}u_h^{n+1}
    =H_{\kappa,\theta}(u_h^n),
    \label{eq:coercive-form}
\end{equation}
where the operator-induced seminorm is defined by
\[
    \dt u_h^{n+1}
    :=\frac{u_h^{n+1}-u_h^n}{\tau}.
\]
Moreover, $\Mfun(\tau\Lop_{\kappa,h})$ is symmetric positive definite.
\end{lemma}

\begin{proof}
Let $L=\Lop_{\kappa,h}$ and $S=e^{-\tau L}$. Since
\[
    \tau\phiOne(\tau L)
    =(I-S)L^{-1},
\]
the displayed identity is understood through spectral functional calculus when $L$ has a zero eigenvalue. More explicitly, on an eigenvector of $L$ with eigenvalue $\lambda\ge0$, both sides act by
\[
    \frac{1-e^{-\tau\lambda}}{\lambda},
\]
with the quotient continuously extended to the value $\tau$ at $\lambda=0$; thus no ordinary inverse of a singular $L$ is required. Equivalently, the identity
\[
    I-e^{-\tau L}=\tau L\phiOne(\tau L)
\]
can be used throughout the calculation. Moreover, $\phiOne(\tau\lambda)>0$ for every $\lambda\ge0$, so $\phiOne(\tau L)$ is invertible even when $L$ is only nonnegative.
Equation~\eqref{eq:etd1} implies
\[
    \dt u_h^{n+1}
    =\phiOne(\tau L)\bigl(H_{\kappa,\theta}(u_h^n)-Lu_h^n\bigr).
\]
Hence
\[
    \phiOne(\tau L)^{-1}\dt u_h^{n+1}+Lu_h^n
    =H_{\kappa,\theta}(u_h^n).
\]
Using $u_h^n=u_h^{n+1}-\tau\dt u_h^{n+1}$ gives
\[
    \left[\phiOne(\tau L)^{-1}-\tau L\right]\dt u_h^{n+1}
    +Lu_h^{n+1}
    =H_{\kappa,\theta}(u_h^n).
\]
For a scalar $z\ge0$,
\[
    \phiOne(z)^{-1}-z
    =\frac{z}{e^z-1}=\Mfun(z).
\]
The spectral theorem therefore yields~\eqref{eq:coercive-form}. Since
$\Mfun(z)>0$ for $z\ge0$, the operator $\Mfun(\tau L)$ is positive definite.
\end{proof}

\section{Maximum-bound principle and unconditional energy stability}

We next show that the \TIFE{} itself provides the information required to preserve an admissible interval.

\begin{assumption}[Inward endpoint condition]
\label{ass:endpoint}
The \TIFE{} density satisfies
\begin{equation}
    F_\theta'(-\beta)\le0,
    \qquad
    F_\theta'(\beta)\ge0.
    \label{eq:endpoint-condition}
\end{equation}
\end{assumption}

\begin{lemma}[Bounded stabilized nonlinearity]
\label{lem:bounded-nonlinearity}
Suppose $F_\theta\in C^2([-\beta,\beta])$,
condition~\eqref{eq:kappa-condition} holds with $\Aset=[-\beta,\beta]$, and
Assumption~\ref{ass:endpoint} is satisfied. Then
\begin{equation}
    |H_{\kappa,\theta}(z)|\le\kappa\beta,
    \qquad |z|\le\beta.
    \label{eq:H-bound}
\end{equation}
\end{lemma}

\begin{proof}
Since
\[
    H_{\kappa,\theta}'(z)
    =\kappa-F_\theta''(z)\ge0,
\]
$H_{\kappa,\theta}$ is nondecreasing on $[-\beta,\beta]$. Hence
\[
    H_{\kappa,\theta}(z)
    \le H_{\kappa,\theta}(\beta)
    =\kappa\beta-F_\theta'(\beta)
    \le\kappa\beta,
\]
and similarly
\[
    H_{\kappa,\theta}(z)
    \ge H_{\kappa,\theta}(-\beta)
    =-\kappa\beta-F_\theta'(-\beta)
    \ge-\kappa\beta.
\]
This proves~\eqref{eq:H-bound}.
\end{proof}

\begin{theorem}[Unconditional maximum-bound principle]
\label{thm:mbp}
Assume Assumption~\ref{ass:positive-semigroup}, Assumption~\ref{ass:endpoint}, and
\[
    \kappa\ge\sup_{|z|\le\beta}F_\theta''(z).
\]
If
\[
    \norm{u_h^0}_{\infty,h}\le\beta,
\]
then the \TIFE-based ETD scheme~\eqref{eq:etd1} satisfies
\begin{equation}
    \boxed{
    \norm{u_h^n}_{\infty,h}\le\beta,
    \qquad n\ge0,
    }
    \label{eq:mbp-result}
\end{equation}
for every $\tau>0$.
\end{theorem}

\begin{proof}
The assertion is true for $n=0$ by hypothesis.
Assume inductively that $\norm{u_h^n}_{\infty,h}\le\beta$. By Assumption~\ref{ass:positive-semigroup},
\[
    \norm{e^{-\tau\Lop_{\kappa,h}}u_h^n}_{\infty,h}
    =e^{-\kappa\tau}
      \norm{e^{-\tau\Lop_h}u_h^n}_{\infty,h}
    \le e^{-\kappa\tau}\beta.
\]
Moreover,
\begin{align*}
    \tau\phiOne(\tau\Lop_{\kappa,h})
    &=\int_0^\tau e^{-(\tau-s)\Lop_{\kappa,h}}\,ds,
\end{align*}
and the integrand is positivity preserving. For $0\le s\le\tau$,
Assumption~\ref{ass:positive-semigroup} and~\eqref{eq:nonnegative-kappa} give
\[
    \norm{e^{-(\tau-s)\Lop_{\kappa,h}}v_h}_{\infty,h}
    \le e^{-\kappa(\tau-s)}\norm{v_h}_{\infty,h}.
\]
This makes the semigroup estimate used in the next line explicit. Therefore, using Lemma~\ref{lem:bounded-nonlinearity},
\begin{align*}
    \left\|
    \tau\phiOne(\tau\Lop_{\kappa,h})
    H_{\kappa,\theta}(u_h^n)
    \right\|_{\infty,h}
    &\le
    \int_0^\tau
    e^{-\kappa(\tau-s)}\,ds\;\kappa\beta
    \\
    &=(1-e^{-\kappa\tau})\beta.
\end{align*}
The last identity is understood by continuity when $\kappa=0$, in which case both sides vanish.
Thus
\[
    \norm{u_h^{n+1}}_{\infty,h}
    \le e^{-\kappa\tau}\beta
    +(1-e^{-\kappa\tau})\beta
    =\beta.
\]
The induction is complete.
\end{proof}

\begin{remark}[Strict separation]
For a singular potential such as the Flory--Huggins energy, the same argument applies on a strictly separated interval
\[
    \Aset_\delta=[-1+\delta,1-\delta],
    \qquad \delta>0,
\]
provided the corresponding inward endpoint condition and curvature bound hold on $\Aset_\delta$.
\end{remark}

\subsection{Unconditional dissipation of the original \TIFE}
\label{sec:energy-dissipation}
Define tfdhe shifted nonlinear energy density
\begin{equation}
    G_{\kappa,\theta}(z)
    :=F_\theta(z)-\frac{\kappa}{2}z^2.
    \label{eq:G-def}
\end{equation}
Then
\begin{equation}
    G_{\kappa,\theta}''(z)
    =F_\theta''(z)-\kappa\le0
    \qquad z\in\Aset.
    \label{eq:G-concave}
\end{equation}
Hence $G_{\kappa,\theta}$ is concave on the admissible set. The \TIFE{} functional admits the exact decomposition
\begin{equation}
    \E_{\theta,h}[v_h]
    =\frac12\innerh{v_h}{\Lop_{\kappa,h}v_h}
     +\innerh{G_{\kappa,\theta}(v_h)}{1}.
    \label{eq:shifted-energy-decomposition}
\end{equation}

\begin{theorem}[Unconditional dissipation of the original \TIFE]
\label{thm:energy}
Let $\Aset\subset\R$ be a convex interval containing $u_h^n$ and $u_h^{n+1}$ pointwise, and assume
\begin{equation}
    F_\theta''(z)\le\kappa,
    \qquad z\in\Aset.
    \label{eq:semiconcavity-condition}
\end{equation}
Then the \TIFE-based ETD scheme~\eqref{eq:etd1} satisfies
\begin{align}
    \E_{\theta,h}[u_h^{n+1}]
    -\E_{\theta,h}[u_h^n]
    \le{}&
    -\tau\innerh{
    \Mfun(\tau\Lop_{\kappa,h})\dt u_h^{n+1}}
    {\dt u_h^{n+1}}
    \nonumber\\
    &-\frac12
    \norm{u_h^{n+1}-u_h^n}_{\Lop_{\kappa,h},h}^2
    \le0,
    \label{eq:energy-dissipation}
\end{align}
where
\[
    \norm{v_h}_{\Lop_{\kappa,h},h}^2
    :=\innerh{v_h}{\Lop_{\kappa,h}v_h}.
\]
In particular, under the assumptions of Theorem~\ref{thm:mbp},
\[
    \boxed{
    \E_{\theta,h}[u_h^{n+1}]
    \le \E_{\theta,h}[u_h^n],
    \qquad \forall\tau>0.
    }
\]
\end{theorem}

\begin{proof}
By Lemma~\ref{lem:coercive-reformulation} and
\[
    H_{\kappa,\theta}(u_h^n)
    =-G_{\kappa,\theta}'(u_h^n),
\]
we have
\[
    \Mfun(\tau\Lop_{\kappa,h})\dt u_h^{n+1}
    +\Lop_{\kappa,h}u_h^{n+1}
    =-G_{\kappa,\theta}'(u_h^n).
\]
Taking the discrete inner product with
\[
    \du u_h^{n+1}:=u_h^{n+1}-u_h^n
    =\tau\dt u_h^{n+1}
\]
and using the symmetry identity
\[
    \innerh{\Lop_{\kappa,h}u_h^{n+1}}
    {u_h^{n+1}-u_h^n}
    =\frac12\left(
      \norm{u_h^{n+1}}_{\Lop_{\kappa,h},h}^2
      -\norm{u_h^n}_{\Lop_{\kappa,h},h}^2
      +\norm{u_h^{n+1}-u_h^n}_{\Lop_{\kappa,h},h}^2
    \right),
\]
gives
\begin{align}
    &\tau\innerh{
    \Mfun(\tau\Lop_{\kappa,h})\dt u_h^{n+1}}
    {\dt u_h^{n+1}}
    \nonumber\\
    &\quad+
    \frac12\left(
    \norm{u_h^{n+1}}_{\Lop_{\kappa,h},h}^2
    -\norm{u_h^n}_{\Lop_{\kappa,h},h}^2
    +\norm{\du u_h^{n+1}}_{\Lop_{\kappa,h},h}^2
    \right)
    \nonumber\\
    &=-\innerh{G_{\kappa,\theta}'(u_h^n)}
    {\du u_h^{n+1}}.
    \label{eq:energy-proof-identity}
\end{align}
Because $G_{\kappa,\theta}$ is concave on $\Aset$,
the supporting-line inequality is valid on every segment contained in $\Aset$. Since $\Aset$ is convex, the line segment joining each pair of nodal values $u_{h,j}^n$ and $u_{h,j}^{n+1}$ remains in $\Aset$. Hence the inequality applies at every node:
\[
    G_{\kappa,\theta}(u_h^{n+1})
    -G_{\kappa,\theta}(u_h^n)
    \le
    G_{\kappa,\theta}'(u_h^n)
    \du u_h^{n+1}
\]
pointwise. Summation in the discrete inner product therefore yields
\[
    -\innerh{G_{\kappa,\theta}'(u_h^n)}
    {\du u_h^{n+1}}
    \le
    -\innerh{
    G_{\kappa,\theta}(u_h^{n+1})
    -G_{\kappa,\theta}(u_h^n)}{1}.
\]
Combining this estimate with~\eqref{eq:energy-proof-identity} and the energy
decomposition~\eqref{eq:shifted-energy-decomposition}
gives~\eqref{eq:energy-dissipation}.
More explicitly, these relations imply
\begin{align*}
    &\E_{\theta,h}[u_h^{n+1}]-\E_{\theta,h}[u_h^n]
      +\tau\innerh{\Mfun(\tau\Lop_{\kappa,h})\dt u_h^{n+1}}
      {\dt u_h^{n+1}}
      +\frac12\norm{\du u_h^{n+1}}_{\Lop_{\kappa,h},h}^2
    \le0,
\end{align*}
which is the first inequality in~\eqref{eq:energy-dissipation} after
rearrangement. The two terms on the right-hand side of that inequality are
nonpositive: $\Mfun(\tau\Lop_{\kappa,h})$ is positive definite by
Lemma~\ref{lem:coercive-reformulation}, whereas $\Lop_{\kappa,h}$ is
nonnegative by~\eqref{eq:nonnegative-kappa}. Thus \emph{unconditional} means
that no time-step restriction is imposed; the two numerical states must still
belong to the stated admissible set.
\end{proof}

\begin{remark}[Role of the maximum-bound principle]
The maximum-bound principle is not intrinsically required by the energy proof. Its role is to provide an invariant interval on which the upper curvature bound $F_\theta''\le\kappa$ is valid. If $F_\theta''$ is globally bounded above, or if another a priori argument confines the numerical solution to a curvature-bounded set, Theorem~\ref{thm:energy} remains valid without an MBP.
\end{remark}

\begin{corollary}[Structure preservation with a learned Laplacian coefficient]
\label{cor:learned-coefficient-structure}
Let $A_h=-\Delta_h$ be symmetric nonnegative and suppose that
$e^{-tA_h}$ is positivity preserving and $\ell_h^\infty$-contractive.  For
any inferred coefficient $a_\alpha\ge0$, set
$\Lop_h=a_\alpha A_h$ and use the
update~\eqref{eq:joint-coefficient-etd1}. Under the endpoint and curvature
assumptions of Theorems~\ref{thm:mbp} and~\ref{thm:energy}, and provided that
$\norm{u_h^0}_{\infty,h}\le\beta$,
\begin{equation}
    \norm{u_h^n}_{\infty,h}\le\beta,
    \qquad
    \E_{\alpha,\theta,h}[u_h^{n+1}]
    \le\E_{\alpha,\theta,h}[u_h^n]
    \label{eq:learned-coefficient-structure}
\end{equation}
for every $\tau>0$.  These properties do not require
$a_\alpha$ to be close to $a_\star$; they concern the dynamics and inferred
energy of the inferred pair $(a_\alpha,F_\theta)$.
\end{corollary}

\begin{proof}
Scaling $A_h$ by the nonnegative number $a_\alpha$ preserves symmetry,
nonnegativity, positivity preservation, and semigroup contractivity.  Hence
all hypotheses involving the linear operator in Theorems~\ref{thm:mbp}
and~\ref{thm:energy} remain valid. Their discrete energy becomes
exactly~\eqref{eq:joint-learned-energy}, which proves the assertion.
\end{proof}

\section{Error analysis with constitutive and coefficient inference errors}

We first compare the fully discrete \TIFE-based ETD solution with the exact solution of the true gradient flow when the linear operator is supplied,
\begin{equation}
    u_t+\Lop u=-F'(u).
    \label{eq:true-pde-error}
\end{equation}
The key point is that the total error naturally separates into discretization error and \TIFE-induced-force inference error.

\subsection{Assumptions for the error estimate}
Let $\Ih$ be the grid restriction operator. We assume that there exists a compact convex interval $\Aset$ such that
\begin{equation}
    u(x,t)\in\Aset,
    \qquad
    u_{h,j}^n\in\Aset,
    \label{eq:common-admissible-set}
\end{equation}
for all $x\in\Omega$, $j\in\mathcal J_h$, $0\le t\le T$, and $t_n\le T$.
In the MBP setting, the second inclusion follows from
Theorem~\ref{thm:mbp}.

Define
\begin{equation}
    \delta_F
    :=\norm{F_\theta'-F'}_{L^\infty(\Aset)},
    \label{eq:inference-error}
\end{equation}
and
\begin{equation}
    L_\theta
    :=\norm{\kappa-F_\theta''}_{L^\infty(\Aset)}.
    \label{eq:tife-Lipschitz}
\end{equation}
We also set
\begin{equation}
    L_H:=\norm{\kappa-F''}_{L^\infty(\Aset)},
    \label{eq:true-Lipschitz}
\end{equation}
which is finite whenever $F\in C^2(\Aset)$ and controls the Lipschitz constant of the stabilized true nonlinearity $H$.
For a family of \TIFE{} models, a uniform bound on $L_\theta$ is required if the final error constant is to be independent of the particular training realization.

We further assume the second-order spatial consistency
\begin{equation}
    \sup_{0\le t\le T}
    \normh{
    \Lop_{\kappa,h}\Ih u(t)
    -\Ih\Lop_\kappa u(t)}
    \le Ch^2,
    \label{eq:spatial-consistency}
\end{equation}
where $\Lop_\kappa=\Lop+\kappa I$, together with sufficient regularity, for example
\begin{equation}
    u\in L^\infty(0,T;H^4(\Omega)),
    \qquad
    u_t\in L^\infty(0,T;L^2(\Omega)).
    \label{eq:regularity}
\end{equation}
For standard second-order finite differences,
condition~\eqref{eq:spatial-consistency} follows from~\eqref{eq:regularity}.
This assertion is intended for the usual Laplacian-type second-order operators
and their standard consistent stencils; for a different spatial operator,
condition~\eqref{eq:spatial-consistency} is an independent assumption that must be
verified for the chosen discretization.
To justify the time-increment estimate after nodal restriction, we additionally impose the restriction-stability condition
\begin{equation}
    \sup_{0\le t\le T}\normh{\Ih u_t(t)}\le C_u,
    \label{eq:restriction-stability}
\end{equation}
with $C_u$ independent of $h$. This condition follows, for example, if
$u_t\in L^\infty(0,T;H^s(\Omega))$ with $s>d/2$ for a standard nodal
restriction. Alternatively, it may be taken as a direct consistency
assumption for the chosen restriction or projection operator. The $L^2$
regularity of $u_t$ in~\eqref{eq:regularity} alone does not, in general,
control its point samples uniformly in $h$.

\subsection{Nonlinear perturbation estimate}
Define the stabilized true and \TIFE-induced nonlinearities
\begin{equation}
    H(z):=\kappa z-F'(z),
    \qquad
    H_\theta(z):=\kappa z-F_\theta'(z).
    \label{eq:H-true-tife}
\end{equation}

\begin{lemma}[Separation of numerical and inference errors]
\label{lem:nonlinear-error}
For any $v,w\in\Aset$,
\begin{equation}
    |H_\theta(v)-H(w)|
    \le L_\theta|v-w|+\delta_F.
    \label{eq:nonlinear-perturbation}
\end{equation}
Consequently, for grid functions $v_h,w_h$ taking values in $\Aset$,
\begin{equation}
    \normh{H_\theta(v_h)-H(w_h)}
    \le L_\theta\normh{v_h-w_h}
    +\normh{1}\,\delta_F.
    \label{eq:nonlinear-grid-bound}
\end{equation}
\end{lemma}

\begin{proof}
Add and subtract $H_\theta(w)$:
\begin{align*}
    H_\theta(v)-H(w)
    &=[H_\theta(v)-H_\theta(w)]
      +[F'(w)-F_\theta'(w)].
\end{align*}
By the mean-value theorem,
\[
    |H_\theta(v)-H_\theta(w)|
    \le L_\theta|v-w|,
\]
and the second term is bounded by $\delta_F$ by definition. This
proves~\eqref{eq:nonlinear-perturbation}; the grid estimate follows by applying
the pointwise inequality at every node and then taking the discrete norm.
\end{proof}

\subsection{Local consistency of the \TIFE-based ETD1 scheme}
Let
\begin{equation}
    S_{\tau,h}:=e^{-\tau\Lop_{\kappa,h}},
    \qquad
    B_{\tau,h}:=\tau\phiOne(\tau\Lop_{\kappa,h}).
\end{equation}
The restricted exact solution $U_h(t):=\Ih u(t)$ satisfies
\begin{equation}
    \frac{d}{dt}U_h(t)
    +\Lop_{\kappa,h}U_h(t)
    =H(U_h(t))+R_h(t),
    \label{eq:restricted-exact}
\end{equation}
where
\begin{equation}
    R_h(t)
    :=\Lop_{\kappa,h}\Ih u(t)
      -\Ih\Lop_\kappa u(t).
\end{equation}
Here we used the nodal identities $\Ih F'(u)=F'(\Ih u)$ and
$\Ih(\kappa u)=\kappa\Ih u$. Indeed, restricting~\eqref{eq:true-pde-error} in
its stabilized form gives
\[
    \frac{d}{dt}U_h(t)+\Ih\Lop_\kappa u(t)=H(U_h(t)),
\]
and adding $\Lop_{\kappa,h}U_h-\Ih\Lop_\kappa u=R_h$
proves~\eqref{eq:restricted-exact}, including the sign of $R_h$.
Variation of constants gives
\begin{equation}
    U_h(t_{n+1})
    =S_{\tau,h}U_h(t_n)
     +B_{\tau,h}H(U_h(t_n))
     +\xi_h^{n+1},
    \label{eq:exact-etd-relation}
\end{equation}
with
\begin{align}
    \xi_h^{n+1}
    :={}&
    \int_0^\tau
    e^{-(\tau-s)\Lop_{\kappa,h}}
    \bigl[H(U_h(t_n+s))-H(U_h(t_n))\bigr]ds
    \nonumber\\
    &+
    \int_0^\tau
    e^{-(\tau-s)\Lop_{\kappa,h}}
    R_h(t_n+s)\,ds.
    \label{eq:local-defect}
\end{align}

\begin{lemma}[Local consistency]
\label{lem:local-consistency}
Assume~\eqref{eq:spatial-consistency}--\eqref{eq:regularity}
and~\eqref{eq:restriction-stability}, and let $F\in C^2(\Aset)$. Then
\begin{equation}
    \normh{\xi_h^{n+1}}
    \le C\tau(\tau+h^2),
    \qquad t_{n+1}\le T.
    \label{eq:local-consistency-bound}
\end{equation}
\end{lemma}

\begin{proof}
Since $\Lop_{\kappa,h}$ is symmetric nonnegative,
\[
    \norm{e^{-t\Lop_{\kappa,h}}}_{h\to h}\le1.
\]
Because $F\in C^2(\Aset)$, $H$ is Lipschitz on $\Aset$.
By~\eqref{eq:restriction-stability},~\eqref{eq:true-Lipschitz}, and the
fundamental theorem of calculus,
\[
    \normh{U_h(t_n+s)-U_h(t_n)}
    \le \int_{t_n}^{t_n+s}\normh{\Ih u_t(r)}\,dr
    \le C_u s,
    \qquad 0\le s\le\tau,
\]
and the constant in the preceding display may be taken as $L_H C_u$, independently of $h$.
Consequently,
\[
    \normh{H(U_h(t_n+s))-H(U_h(t_n))}
    \le L_H C_u s.
\]
Therefore the first integral in~\eqref{eq:local-defect} is $O(\tau^2)$. By
condition~\eqref{eq:spatial-consistency}, the second integral is
$O(\tau h^2)$. This
proves~\eqref{eq:local-consistency-bound}.
Indeed, semigroup contractivity yields
\[
    \left\|\int_0^\tau e^{-(\tau-s)\Lop_{\kappa,h}}
    [H(U_h(t_n+s))-H(U_h(t_n))]\,ds\right\|_h
    \le \int_0^\tau L_H C_u s\,ds
    =\frac12 L_H C_u\tau^2,
\]
whereas the residual term satisfies
\[
    \left\|\int_0^\tau e^{-(\tau-s)\Lop_{\kappa,h}}R_h(t_n+s)\,ds\right\|_h
    \le\int_0^\tau\normh{R_h(t_n+s)}\,ds
    \le C\tau h^2.
\]
\end{proof}

\subsection{Fully discrete error estimate}
Define
\begin{equation}
    e_h^n:=u_h^n-U_h(t_n).
    \label{eq:error-def}
\end{equation}
Subtracting~\eqref{eq:exact-etd-relation} from the numerical scheme gives
\begin{equation}
    e_h^{n+1}
    =S_{\tau,h}e_h^n
    +B_{\tau,h}
    \bigl[H_\theta(u_h^n)-H(U_h(t_n))\bigr]
    -\xi_h^{n+1}.
    \label{eq:error-equation}
\end{equation}

\begin{theorem}[Fully discrete error with force inference error]
\label{thm:error}
Assume the common-interval condition~\eqref{eq:common-admissible-set}, the
spatial-consistency condition~\eqref{eq:spatial-consistency}, the
regularity condition~\eqref{eq:regularity}, and the restriction-stability
condition~\eqref{eq:restriction-stability}, with
$F,F_\theta\in C^2(\Aset)$ and
\[
    \delta_F
    =\norm{F_\theta'-F'}_{L^\infty(\Aset)}<\infty,
    \qquad
    L_\theta
    =\norm{\kappa-F_\theta''}_{L^\infty(\Aset)}<\infty.
\]
Then, for $0<\tau\le\tau_0$ and $0<h\le h_0$ with no coupling condition between $\tau$ and $h$,
\begin{equation}
    \boxed{
    \max_{0\le n\le N_T}
    \normh{u_h^n-\Ih u(t_n)}
    \le
    C_T\left(
    \normh{u_h^0-\Ih u(0)}
    +\tau+h^2+\delta_F
    \right),
    }
    \label{eq:main-error-estimate}
\end{equation}
where $N_T\tau\le T$. If, in addition,
\begin{equation}
    \normh{u_h^0-\Ih u(0)}\le Ch^2,
\end{equation}
then
\begin{equation}
    \boxed{
    \max_{0\le n\le N_T}
    \normh{u_h^n-\Ih u(t_n)}
    \le C_T(\tau+h^2+\delta_F).
    }
    \label{eq:simplified-error-estimate}
\end{equation}
If a family of \TIFE{} densities is considered, the constant $C_T$ is uniform provided $L_\theta$ and the relevant regularity bounds are uniformly bounded.
\end{theorem}

\begin{proof}
Since $\Lop_{\kappa,h}$ is symmetric nonnegative,
\begin{equation}
    \norm{S_{\tau,h}}_{h\to h}\le1.
\end{equation}
Furthermore, for $z\ge0$,
\[
    0<\phiOne(z)\le1,
\]
so
\begin{equation}
    \norm{B_{\tau,h}}_{h\to h}\le\tau.
    \label{eq:B-bound}
\end{equation}
Using~\eqref{eq:error-equation}, Lemmas~\ref{lem:nonlinear-error}
and~\ref{lem:local-consistency}, and~\eqref{eq:B-bound}, we obtain
\begin{align}
    \normh{e_h^{n+1}}
    &\le
    \normh{e_h^n}
    +\tau\left(
    L_\theta\normh{e_h^n}
    +\normh{1}\delta_F
    \right)
    +C\tau(\tau+h^2)
    \nonumber\\
    &\le
    (1+C\tau)\normh{e_h^n}
    +C\tau(\delta_F+\tau+h^2).
    \label{eq:error-recursion}
\end{align}
Applying the discrete Gronwall inequality on $0\le t_n\le T$ yields
\begin{equation}
    \normh{e_h^n}
    \le
    C_T\left(
    \normh{e_h^0}
    +\delta_F+\tau+h^2
    \right).
\end{equation}
For completeness, this follows by iterating~\eqref{eq:error-recursion}. With
$q:=\delta_F+\tau+h^2$,
\[
    \normh{e_h^n}
    \le (1+C\tau)^n\normh{e_h^0}
      +C\tau q\sum_{j=0}^{n-1}(1+C\tau)^j
    \le e^{CT}\bigl(\normh{e_h^0}+C T q\bigr),
\]
where the last inequality also covers $C=0$ by continuity. Hence the constant is independent of $h$, $\tau$, and $\delta_F$; it may depend on $T$, $\Aset$, the uniform bounds for $L_\theta$ and $L_H$, the spatial-consistency constant, $C_u$, and the domain through the uniform bound for $\normh{1}$.
Taking the maximum over $0\le n\le N_T$ gives~\eqref{eq:main-error-estimate};
the second conclusion follows from the assumed initial projection error.
\end{proof}

\begin{corollary}[Training-residual-to-rollout estimate]
\label{cor:data-to-rollout}
Assume the hypotheses of Corollary~\ref{cor:uniform-force-from-training} and
Theorem~\ref{thm:error}, and suppose that the exact and numerical rollout
states remain in $\Aset_{\rm data}$.  Then
\begin{equation}
    \boxed{
    \max_{0\le n\le N_T}
    \normh{u_h^n-\Ih u(t_n)}
    \le C_T\left(
      \normh{u_h^0-\Ih u(0)}
      +\tau+h^2+\mathfrak E_{\rm data}
    \right),}
    \label{eq:data-to-rollout-estimate}
\end{equation}
where $\mathfrak E_{\rm data}$ is given
by~\eqref{eq:data-force-estimator}. Hence the learned-model rollout error is
controlled by the achieved ETD response residual, observation perturbations,
the finite-$\Delta t$ freezing defect, and state-space coverage, in addition
to the deployment discretization errors.  The displayed formula uses the same
grid symbol for training and deployment.  If different grids are used,
$\mathfrak E_{\rm data}$ contains the training-grid constants
$\beta_{\Delta t,h_{\rm tr}}$ and $\omega_{\min,h_{\rm tr}}$, whereas the
spatial consistency term is $h_{\rm dep}^2$.
\end{corollary}

\begin{proof}
Corollary~\ref{cor:uniform-force-from-training} gives
$\delta_F\le\mathfrak E_{\rm data}$ on the common admissible
interval. Substitution into~\eqref{eq:main-error-estimate}
proves~\eqref{eq:data-to-rollout-estimate}.
\end{proof}

\begin{remark}[Interpretation of the error estimate]
Estimate~\eqref{eq:simplified-error-estimate} separates three conceptually
different errors:
\[
    \boxed{
    \text{total error}
    =O(\tau)
    +O(h^2)
    +O(\delta_F).
    }
\]
Here $O(\tau)$ and $O(h^2)$ are the numerical discretization errors, whereas $\delta_F$ is the inference error in the \TIFE-induced variational force.  Corollary~\ref{cor:data-to-rollout} further resolves this inference term into the optimized response residual, observation error, ETD1 freezing bias, response-inversion conditioning, and state-space coverage.  Structure preservation itself does not require $\delta_F$ to be small: even an imperfect \TIFE{} density is evolved by the \TIFE-based ETD scheme, which preserves the admissible interval and dissipates its own \TIFE{} functional provided the curvature and endpoint conditions are satisfied.
\end{remark}

\subsection{Additional error from the inferred Laplacian coefficient}
\label{sec:coefficient-error}

We finally compare the jointly inferred model with the true equation
\begin{equation}
    u_t+a_\star Au=-F'(u),
    \qquad A=-\Delta,
    \label{eq:true-coefficient-model}
\end{equation}
and define
\begin{equation}
    \delta_a:=|a_\alpha-a_\star|.
    \label{eq:coefficient-error}
\end{equation}
Assume that $0\le a_\alpha\le a_{\max}$ and that the base discretization
satisfies
\begin{equation}
    \sup_{0\le t\le T}
    \normh{A_h\Ih u(t)-\Ih Au(t)}\le Ch^2,
    \qquad
    \sup_{0\le t\le T}\normh{\Ih Au(t)}\le C_A.
    \label{eq:base-operator-consistency}
\end{equation}

\begin{theorem}[Fully discrete error with coefficient and force inference]
\label{thm:coefficient-force-error}
Assume the common-interval, regularity, and restriction-stability hypotheses
of Theorem~\ref{thm:error}, together
with~\eqref{eq:base-operator-consistency}. Let $u_h^n$ be generated
by~\eqref{eq:joint-coefficient-etd1}. Then, for $N_T\tau\le T$,
\begin{equation}
    \boxed{
    \max_{0\le n\le N_T}
    \normh{u_h^n-\Ih u(t_n)}
    \le C_T\left(
       \normh{u_h^0-\Ih u(0)}
       +\tau+h^2+\delta_F+\delta_a
    \right).}
    \label{eq:coefficient-force-error-bound}
\end{equation}
The constant is uniform for $a_\alpha$ in a fixed bounded interval and for a
family of learned potentials with uniformly bounded $L_\theta$.
\end{theorem}

\begin{proof}
With $U_h(t)=\Ih u(t)$ and
$\Lop_{\alpha,\kappa,h}=a_\alpha A_h+\kappa I$, the restricted true solution
satisfies
\begin{equation}
    \frac{d}{dt}U_h(t)
    +\Lop_{\alpha,\kappa,h}U_h(t)
    =H(U_h(t))+R_{\alpha,h}(t),
    \label{eq:coefficient-restricted-exact}
\end{equation}
where
\begin{align}
    R_{\alpha,h}(t)
    &:={}
    a_\alpha A_h\Ih u(t)-a_\star\Ih Au(t)
    \nonumber\\
    &=a_\alpha
      \bigl[A_h\Ih u(t)-\Ih Au(t)\bigr]
      +(a_\alpha-a_\star)\Ih Au(t).
    \label{eq:coefficient-residual}
\end{align}
Hence~\eqref{eq:base-operator-consistency} gives
\begin{equation}
    \normh{R_{\alpha,h}(t)}
    \le C(h^2+\delta_a).
    \label{eq:coefficient-residual-bound}
\end{equation}
Applying variation of constants to~\eqref{eq:coefficient-restricted-exact}
and freezing $H$ at $t_n$ gives the same relation
as~\eqref{eq:exact-etd-relation}, with
$\Lop_{\kappa,h}$ replaced by $\Lop_{\alpha,\kappa,h}$.  Semigroup
contractivity, the time-increment estimate,
and~\eqref{eq:coefficient-residual-bound} yield the local bound
\begin{equation}
    \normh{\xi_{\alpha,h}^{n+1}}
    \le C\tau(\tau+h^2+\delta_a).
    \label{eq:coefficient-local-defect}
\end{equation}
Subtracting~\eqref{eq:joint-coefficient-etd1}, using
Lemma~\ref{lem:nonlinear-error}, and noting that
$\norm{\tau\phiOne(\tau\Lop_{\alpha,\kappa,h})}_{h\to h}\le\tau$, we obtain
\begin{equation}
    \normh{e_h^{n+1}}
    \le(1+C\tau)\normh{e_h^n}
      +C\tau(\tau+h^2+\delta_F+\delta_a).
\end{equation}
The discrete Gronwall inequality
proves~\eqref{eq:coefficient-force-error-bound}.
\end{proof}

\begin{remark}[From joint training to rollout]
Theorem~\ref{thm:joint-response-error} controls $\delta_a$ and the empirical
force error under Assumption~\ref{ass:joint-identifiability}.  If the sampled
force error is extended to the rollout interval by the fill-distance argument
of Corollary~\ref{cor:uniform-force-from-training}, then
Theorem~\ref{thm:coefficient-force-error} closes the joint chain
\[
    \text{state-pair residual}
    \longrightarrow (\delta_a,\delta_F)
    \longrightarrow \text{rollout error}.
\]
This implication is conditional on joint identifiability and on the rollout
remaining inside the data-supported interval.
\end{remark}




\section{Numerical experiments for data-driven potential identification}
\label{sec:numerical-experiments}

This section evaluates the fixed-coefficient specialization of the proposed
framework: the Laplacian coefficient is supplied to the ETD response, whereas
the bulk potential and its derivative are hidden from the training procedure.
The joint coefficient--energy formulation of
Section~\ref{sec:joint-coefficient-energy} uses the same state-pair data but is
not assigned synthetic recovery values in the present tables; those tables
therefore validate potential identification, not joint coefficient
identification.  The objectives are fourfold.  First, we examine whether
the unknown potential can be identified from short-time state transitions
without using force or energy labels.  Second, we test whether the identified
potential can be reused for unseen initial conditions and substantially longer
prediction horizons.  Third, we investigate whether the combination of the
learned potential and a curvature-informed stabilized ETD scheme retains the
maximum-bound principle (MBP) and dissipates the unmodified learned energy.
Finally, we study reproducibility over five random seeds, data efficiency,
noise robustness, time-step transfer, and a non-polynomial logarithmic
potential.

It is important to distinguish the present learning task from black-box
solution prediction.  The neural network does not approximate the full
space--time solution map $(x,t)\mapsto u(x,t)$.  Instead, it approximates the
scalar constitutive function $u\mapsto F_\theta(u)$.  The known spatial
operator enters both the response-space training objective and the subsequent
ETD evolution. Theorems~\ref{thm:response-to-empirical-force}
and~\ref{thm:error} quantify these two stages, while
Corollary~\ref{cor:data-to-rollout} links the achieved response residual to
the rollout error.  Consequently, one identified potential can be reused for
different initial conditions, grids, time steps, and final times without
retraining, provided that the deployed states remain in the data-supported
interval.

\subsection{Model problem and state-only learning}
\label{subsec:learning-formulation}

We first consider the two-dimensional Allen--Cahn equation on the periodic
domain $\Omega=(0,2\pi)^2$,
\begin{equation}
  u_t=\varepsilon^2\Delta u-F'(u),
  \qquad x\in\Omega,\quad 0<t\le T,
  \label{eq:ac-model-experiment}
\end{equation}
where the exact double-well potential
\begin{equation}
  F(u)=\frac14(u^2-1)^2.
  \label{eq:ac-exact-potential}
\end{equation}
is used only to generate synthetic observations and reference solutions.  In
particular, neither $F$, $F'$, nor samples of the nonlinear force
$N(u)=-F'(u)$ are supplied as training labels.

Let $\Lop_{0,h}=-\varepsilon^2\Delta_h$, where $\Delta_h$ is the periodic
second-order finite-difference Laplacian, and define
\begin{equation}
  S_0(\tau)=e^{-\tau\Lop_{0,h}},
  \qquad
  B_0(\tau)=\tau\phiOne(\tau\Lop_{0,h})
  =\Lop_{0,h}^{-1}\bigl(I-S_0(\tau)\bigr),
  \label{eq:unstabilized-etd-operators}
\end{equation}
where the final expression is interpreted by continuity on the zero Fourier
mode. Given $N_{\rm pair}$ state pairs
$\{(U_j,V_j)\}_{j=1}^{N_{\rm pair}}$ with $V_j\approx u(t_j+\tau)$ and
$U_j=u(t_j)$, the known linear contribution is removed to form the discrete
Duhamel target
\begin{equation}
  D_j=V_j-S_0(\tau)U_j.
  \label{eq:duhamel-target}
\end{equation}
The neural potential induces the nonlinear force
$N_\theta(u)=-F_\theta'(u)$.  It is trained through
\begin{equation}
\begin{split}
  \mathcal J_{\rm train}(\theta)
  ={}&\frac{1}{N_{\rm pair}\sigma_D^2}
  \sum_{j=1}^{N_{\rm pair}}
  \left\|B_0(\tau)N_\theta(U_j)-D_j\right\|_h^2 \\
  &+\frac{0.1}{N_{\rm pair}\sigma_V^2}
  \sum_{j=1}^{N_{\rm pair}}
  \left\|S_0(\tau)U_j+B_0(\tau)N_\theta(U_j)-V_j\right\|_h^2,
  \label{eq:state-only-loss}
\end{split}
\end{equation}
where $\sigma_D$ and $\sigma_V$ are the root-mean-square scales of the
Duhamel targets and next states, respectively.  The normalization prevents
the smaller nonlinear Duhamel contribution from being overwhelmed by the
state magnitude.  Since
$S_0(\tau)U_j+B_0(\tau)N_\theta(U_j)-V_j
=B_0(\tau)N_\theta(U_j)-D_j$, the two terms are differently normalized
copies of the same response residual.  The second term is retained to report
the implemented objective, but it supplies no independent identifying
information beyond the first term.

\subsection{A minimally constrained general potential}
\label{subsec:endpoint-tanh}

To avoid encoding the exact quartic double-well form, we use a general tanh
multilayer perceptron $R_\theta$.  Only the known invariant interval
$[-\beta,\beta]$ and the stationarity of its endpoints are imposed.  Let
\begin{equation}
  s_-=R_\theta'(-\beta),\qquad
  s_+=R_\theta'(\beta),\qquad
  c_1=\frac{s_++s_-}{2},\qquad
  c_2=\frac{s_+-s_-}{4\beta}.
  \label{eq:endpoint-correction-coefficients}
\end{equation}
We define
\begin{equation}
  F_\theta(u)=R_\theta(u)-R_\theta(0)-c_1u-c_2u^2.
  \label{eq:endpoint-tanh-potential}
\end{equation}
Then $F_\theta(0)=0$ fixes the otherwise undetermined additive gauge, and
\begin{equation}
  F_\theta'(-\beta)=F_\theta'(\beta)=0.
  \label{eq:endpoint-stationarity}
\end{equation}
holds by construction.  No symmetry, positivity, polynomial degree, quartic
factorization, or equality of well depths is imposed.  Thus, the network is
substantially more general than the factorized ansatz
$(u^2-\beta^2)^2q_\theta(u)$ used in the preliminary ablation.

After training, let $\mathcal Q_\beta$ be a dense grid on
$[-\beta,\beta]$.  We select the stabilization parameter from the sampled
learned curvature according to
\begin{equation}
  \widehat L_{\theta,\beta}^{+}
  :=\max_{z\in\mathcal Q_\beta}F_\theta''(z),
  \qquad
  \kappa=\max\{0,\widehat L_{\theta,\beta}^{+}\}+\delta_\kappa,
  \qquad \delta_\kappa>0,
  \label{eq:kappa-selection}
\end{equation}
This sampled maximum is an empirical candidate for the upper curvature bound
in~\eqref{eq:kappa-condition}; a rigorous invocation of the theorems requires
certifying the bound on the whole interval.  With
\begin{equation}
  \Lop_{\kappa,h}=\Lop_{0,h}+\kappa I,
  \qquad
  H_{\kappa,\theta}(u)=\kappa u-F_\theta'(u),
  \label{eq:stabilized-splitting}
\end{equation}
the learned dynamics are advanced by
\begin{equation}
  u_h^{n+1}
  =e^{-\tau\Lop_{\kappa,h}}u_h^n
  +\tau\phiOne(\tau\Lop_{\kappa,h})
  H_{\kappa,\theta}(u_h^n).
  \label{eq:learned-stabilized-etd1}
\end{equation}
If the candidate $\kappa$ bounds $F_\theta''$ throughout
$[-\beta,\beta]$, then $H_{\kappa,\theta}$ is monotone and maps the two
endpoints according to $\pm\beta\mapsto\pm\kappa\beta$.  The theory then gives
the MBP and direct dissipation of the unmodified learned energy
\begin{equation}
  \E_{\theta,h}[u]
  =\frac12\innerh{u}{\Lop_{0,h}u}
  +\innerh{F_\theta(u)}{1},
  \label{eq:learned-original-energy}
\end{equation}
without an SAV or GSAV auxiliary variable.  In the experiments we also check
the discrete inequality
\begin{equation}
\begin{split}
  \E_{\theta,h}[u_h^{n+1}]-\E_{\theta,h}[u_h^n]
  &+\tau\innerh{\dt u_h^{n+1}}
  {\Mfun(\tau\Lop_{\kappa,h})\dt u_h^{n+1}} \\
  &+\frac12\innerh{u_h^{n+1}-u_h^n}
  {\Lop_{\kappa,h}(u_h^{n+1}-u_h^n)}\le0,
  \label{eq:checked-energy-inequality}
\end{split}
\end{equation}
where $\dt u_h^{n+1}=(u_h^{n+1}-u_h^n)/\tau$ and $\Mfun$ is defined
in~\eqref{eq:M-function}. The experiments report both the maximum bound and
the residual of this inequality directly.

\subsection{Experimental setup and error metrics}
\label{subsec:experimental-configuration}

The training trajectories contain low- and high-frequency Fourier fields,
droplets, and stripes.  Their amplitudes are selected from
$\{0.35,0.65,0.85,0.98\}\beta$, so the observations cover both the interior
of the admissible interval and states close to its boundary.  Unless stated
otherwise, reference state pairs for the Allen--Cahn equation are generated
by RK4 with multiple internal substeps over each observation interval.  The
exact formulas in~\eqref{eq:ac-exact-potential} are used only for reference
generation and post-training diagnostics.

For states and scalar constitutive functions, we report the relative discrete
$L^2$ error
\begin{equation}
  e_{\rm rel}(v,w)=\frac{\|v-w\|_2}{\|w\|_2+10^{-14}}.
  \label{eq:relative-error-metric}
\end{equation}
Potential errors are computed after aligning the additive gauge at $u=0$.
The principal configurations are summarized in
Table~\ref{tab:experiment-configurations}.

\begin{table}[htbp]
  \centering
  \caption{Principal experimental configurations.  The test horizons exceed
  the training horizons in every long-time experiment.}
  \label{tab:experiment-configurations}
  \small
  \resizebox{\linewidth}{!}{%
  \begin{tabular}{lcccccc}
    \toprule
    Experiment & $N^2$ & $\varepsilon$ & $\tau$ & $T_{\rm train}$
    & Training data & $T_{\rm test}$ \\
    \midrule
    Potential identification & $48^2$ & $0.1$ & $0.02$ & $1.2$
      & $20$ trajectories, $1200$ pairs & $4,10,20,50$ \\
    Endpoint-constrained structure test & $48^2$ & $0.1$ & $0.02$ & $1.2$
      & $20$ trajectories, $1200$ pairs & $4$ \\
    Five-seed benchmark & $24^2$ & $0.1$ & $0.02$ & $0.8$
      & $160$, $320$, or $640$ pairs & $2$ \\
    Flory--Huggins test & $48^2$ & $0.08$ & $0.02$ & $1.2$
      & $20$ trajectories, $1200$ pairs & $20$ \\
    \bottomrule
  \end{tabular}}
\end{table}

The main $48^2$ Allen--Cahn networks have three hidden tanh layers of width
$64$ and are trained for $200$ epochs with Adam.  The five-seed study uses
three hidden layers of width $32$, $80$ epochs, and the seeds
$\{0,1,2,3,4\}$.  Validation trajectories are independent of the training
trajectories.

\FloatBarrier

\subsection{Recovery of the potential and nonlinear force}
\label{subsec:potential-recovery}

We first use a gauge-fixed but otherwise unconstrained tanh potential network
to isolate the benefit of learning an energy primitive rather than a direct
state-to-state map.  Over the state interval visited by the training data, the
identified potential, force, and curvature attain the errors listed in
Table~\ref{tab:identification-errors}.  The recovery of $F''$ is the most
sensitive diagnostic because it differentiates the learned network twice;
nevertheless, its relative error remains approximately $1.25\%$.

\begin{table}[htbp]
  \centering
  \caption{Constitutive-function identification errors on the observed state
  support $[-0.99805,0.99806]$.}
  \label{tab:identification-errors}
  \begin{tabular}{lc}
    \toprule
    Quantity & Relative $L^2$ error \\
    \midrule
    Potential $F_\theta$ & $1.7583\times10^{-3}$ \\
    Nonlinear force $-F_\theta'$ & $4.3604\times10^{-3}$ \\
    Curvature $F_\theta''$ & $1.2501\times10^{-2}$ \\
    \bottomrule
  \end{tabular}
\end{table}

Figure~\ref{fig:recovered-potential-force} shows that the network recovers not
only the force governing the evolution, but also an interpretable bulk energy
density.  This distinction is useful for selecting $\kappa$ and for monitoring
the learned energy during long-time integration.

\begin{figure}[htbp]
  \centering
  \includegraphics[width=\linewidth]{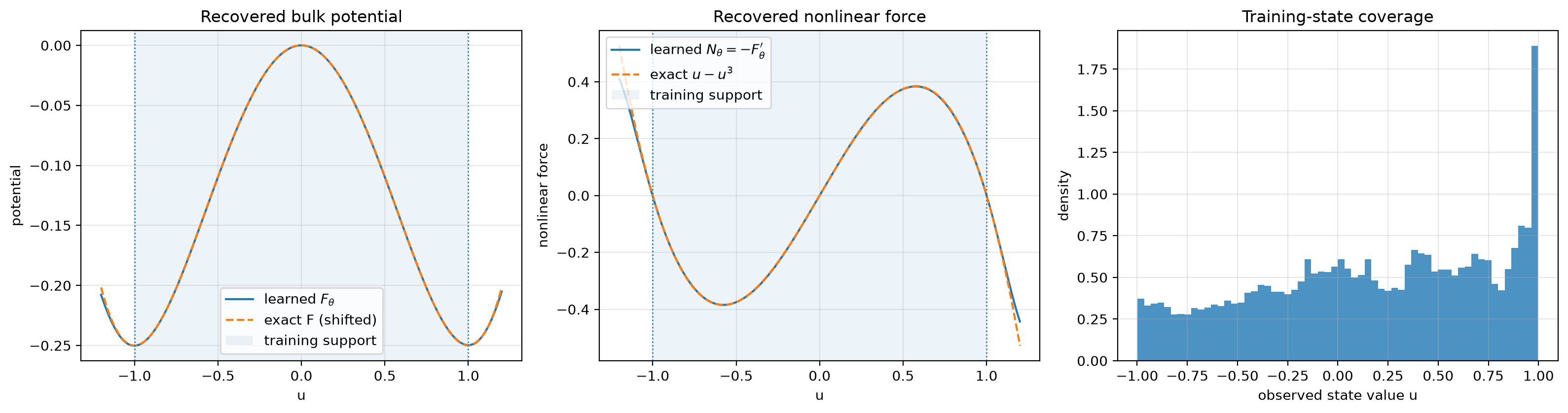}
  \caption{Recovered Allen--Cahn bulk potential and nonlinear force, together
  with the empirical distribution of observed state values.  The exact
  curves are used only for post-training diagnostics.}
  \label{fig:recovered-potential-force}
\end{figure}

For one representative unseen trajectory at $T=4$, the ETD1 method using the
exact nonlinear force has final error $1.3304\times10^{-3}$, while replacing
the force by $-F_\theta'$ gives $1.8493\times10^{-3}$.  Thus, the learned
potential adds only a moderate error to the temporal and spatial
discretization error of the oracle ETD1 method.

\FloatBarrier

\subsection{Unseen initial conditions and long-time prediction}
\label{subsec:long-time-prediction}

The learned potential is next frozen and tested on $12$ initial conditions
that are not used during training.  These tests mix high-frequency Fourier,
droplet, and stripe configurations.  At $T=4$, the mean rollout error is
$1.1974\times10^{-3}$ with standard deviation $4.4248\times10^{-4}$; see
Table~\ref{tab:unseen-initial-conditions}.  The corresponding force error on
the visited states is approximately $0.63\%$.

\begin{table}[htbp]
  \centering
  \caption{Results over $12$ unseen initial conditions at $T=4$.  Values are
  the mean $\pm$ standard deviation.}
  \label{tab:unseen-initial-conditions}
  \begin{tabular}{lc}
    \toprule
    Diagnostic & Mean $\pm$ standard deviation \\
    \midrule
    Learned-force error &
      $6.3227\times10^{-3}\ \pm\ 2.4020\times10^{-3}$ \\
    Learned-potential ETD1 rollout error &
      $1.1974\times10^{-3}\ \pm\ 4.4248\times10^{-4}$ \\
    Exact-force ETD1 rollout error &
      $5.3344\times10^{-4}\ \pm\ 6.5449\times10^{-4}$ \\
    \bottomrule
  \end{tabular}
\end{table}

To evaluate temporal extrapolation, we train only on $0\le t\le1.2$ and
continue the learned dynamics without retraining to $T=10$, $20$, and $50$.
The final errors are reported in Table~\ref{tab:long-time-errors}.  Although
the learned-model error accumulates with time, the solution remains stable and
the relative error is below $0.8\%$ even at $T=50$, more than forty times the
training horizon.

\begin{table}[htbp]
  \centering
  \caption{Long-time extrapolation from a training horizon of $T_{\rm
  train}=1.2$.  Exact-force ETD1 is an oracle discretization baseline.}
  \label{tab:long-time-errors}
  \begin{tabular}{ccc}
    \toprule
    Final time & Learned-potential ETD1 & Exact-force ETD1 \\
    \midrule
    $10$ & $1.5400\times10^{-3}$ & $4.4308\times10^{-4}$ \\
    $20$ & $2.9377\times10^{-3}$ & $9.5322\times10^{-4}$ \\
    $50$ & $7.6535\times10^{-3}$ & $2.7044\times10^{-3}$ \\
    \bottomrule
  \end{tabular}
\end{table}

\begin{figure}[htbp]
  \centering
  \includegraphics[width=\linewidth]{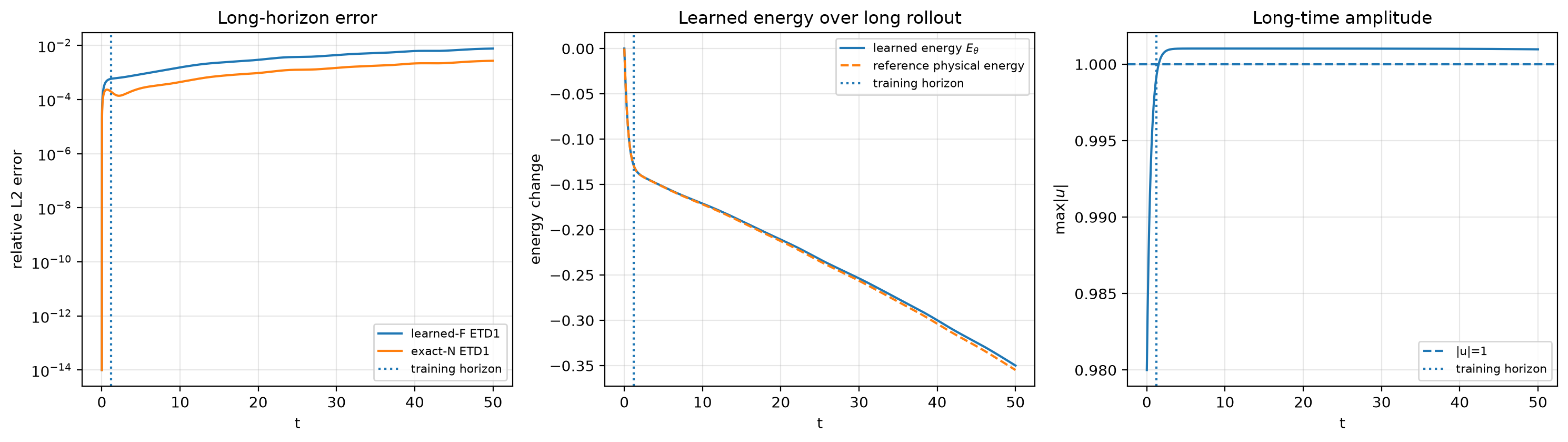}
  \caption{Error evolution and energy diagnostics in the long-time
  extrapolation experiment.  The same learned potential is used at every
  final time.}
  \label{fig:long-time-extrapolation}
\end{figure}

Figure~\ref{fig:long-time-phase-fields} gives the corresponding phase-field
snapshots.  The learned evolution correctly reproduces the slow shrinkage of
the circular inclusion and its interface location up to $T=50$.  The visual
agreement is consistent with the quantitative errors in
Table~\ref{tab:long-time-errors}; in particular, the long-time prediction does
not develop spurious oscillations or an incorrect phase topology.

\begin{figure}[htbp]
  \centering
  \includegraphics[width=\linewidth]{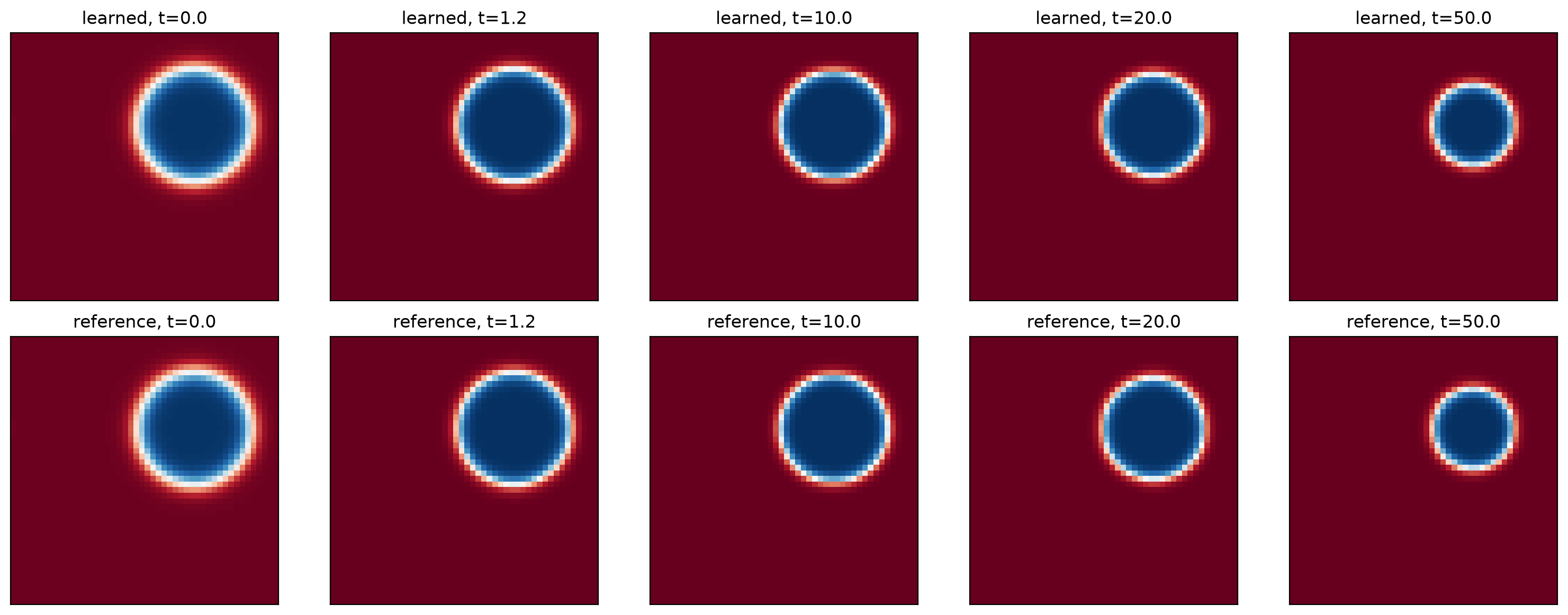}
  \caption{Allen--Cahn phase-field evolution far beyond the training horizon.
  Top: learned-potential ETD1.  Bottom: the fine-step reference solution.  The
  potential is trained only on $0\le t\le1.2$ and is frozen during the
  rollout to $T=50$.}
  \label{fig:long-time-phase-fields}
\end{figure}

\FloatBarrier

\subsection{Comparison with matched-size learned representations}
\label{subsec:baseline-comparison}

We compare three learned representations trained from the same state pairs:
(i) the potential network $F_\theta$, (ii) a direct-force network that learns
$N_\theta(u)$ without an energy primitive, and (iii) a black-box residual FNO
that learns the full state increment.  Their parameter counts are $8513$,
$8513$, and $8529$, respectively.  The exact-force ETD1 method is also
reported as an oracle reference.  The first two representations retain the
known ETD update and learn only its constitutive input, whereas the
residual FNO directly learns a full-field state increment.  The potential
model further constrains that constitutive input to be the derivative of a
scalar energy.  Table~\ref{tab:matched-baselines} shows the mean error over the
same $12$ unseen initial conditions.

For the time-step-transfer tests, the residual FNO uses the explicit
time-rescaling convention
\begin{equation}
    U^{n+1}=U^n+\frac{\tau}{\tau_{\rm train}}
    \mathcal R_{\vartheta}^{\rm FNO}(U^n),
    \label{eq:time-rescaled-fno}
\end{equation}
where $\mathcal R_{\vartheta}^{\rm FNO}$ is trained at
$\tau_{\rm train}=0.02$. This convention changes the increment magnitude but
does not retrain the network.

\begin{table}[htbp]
  \centering
  \caption{Matched-parameter comparison at $T=4$.  One residual-FNO rollout
  failed the finite-value stability check.}
  \label{tab:matched-baselines}
  \begin{tabular}{lccc}
    \toprule
    Method & Parameters & Final relative error & Failures \\
    \midrule
    Exact-force ETD1 & -- &
      $9.2145\times10^{-4}\pm1.1409\times10^{-3}$ & $0/12$ \\
    Potential ETD1 & $8513$ &
      $1.4776\times10^{-3}\pm7.4390\times10^{-4}$ & $0/12$ \\
    Direct-force ETD1 & $8513$ &
      $3.5970\times10^{-3}\pm1.4993\times10^{-3}$ & $0/12$ \\
    Residual FNO & $8529$ &
      $4.1698\times10^{-1}\pm4.6460\times10^{-1}$ & $1/12$ \\
    \bottomrule
  \end{tabular}
\end{table}

In this experiment, learning the potential reduces the mean error of the
direct-force model by approximately $59\%$.  It is also substantially more
accurate and robust than the matched-size black-box residual FNO baseline. The result
supports the central design choice: learning only the missing scalar
constitutive mechanism is more data-efficient than relearning the known
spatial evolution.  The comparison is specific to the present data,
architectures, and parameter budget and is not intended as a universal claim
against all neural-operator variants.

\begin{figure}[htbp]
  \centering
  \includegraphics[width=0.72\linewidth]{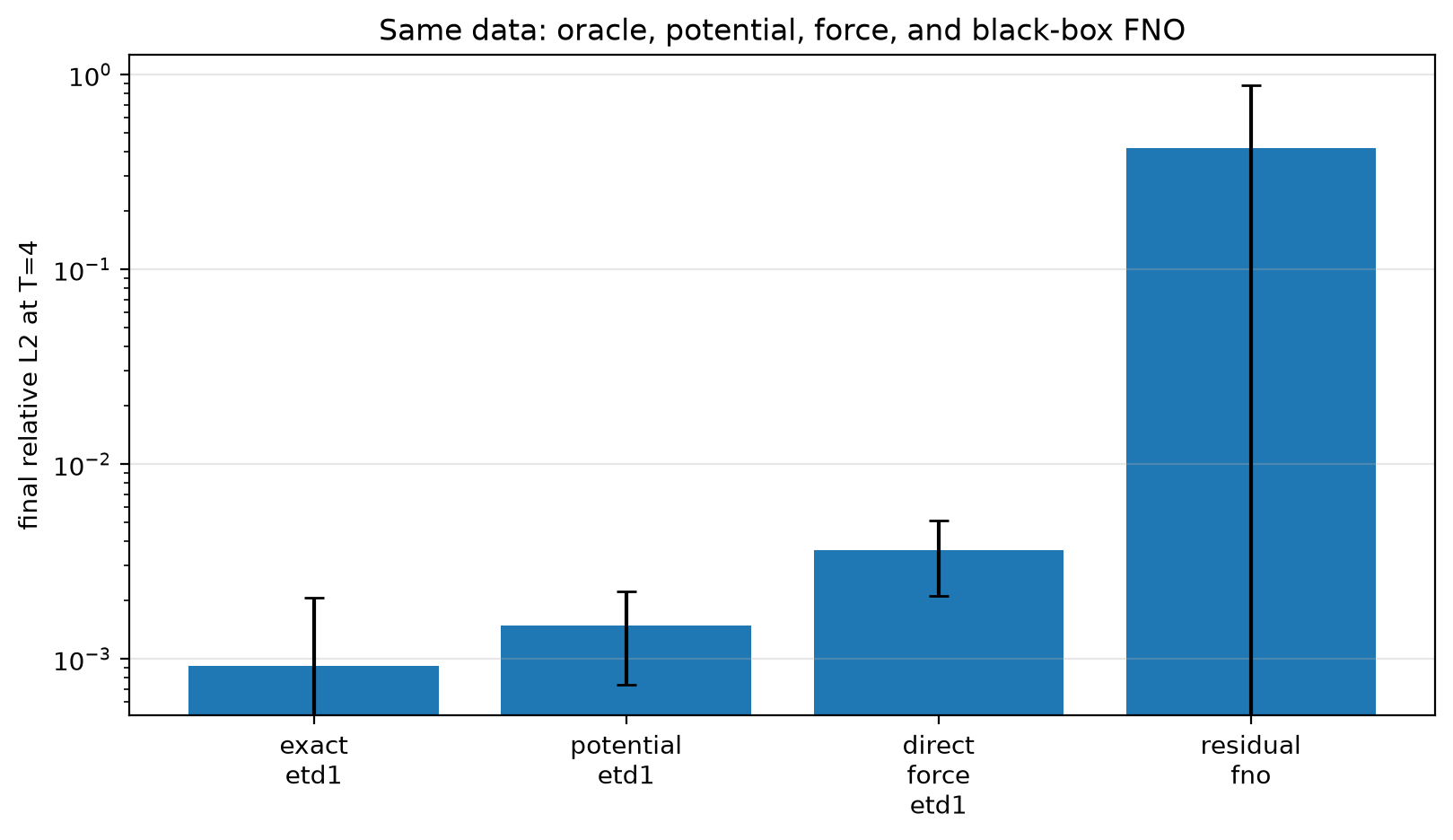}
  \caption{Matched-parameter comparison of the oracle ETD1 method, the
  potential representation, direct-force learning, and a residual FNO.}
  \label{fig:matched-baselines}
\end{figure}

\FloatBarrier

\subsection{Transfer without retraining}
\label{subsec:transfer-results}

The scalar potential does not explicitly depend on the spatial grid or time
step.  We therefore freeze all learned weights and vary one numerical or
physical parameter at a time.  Only the known ETD multipliers are rebuilt.
Table~\ref{tab:no-retraining-transfer} reports the mean final error at $T=2$.

\begin{table}[htbp]
  \centering
  \caption{No-retraining transfer.  The training configuration is
  $(N,\varepsilon,\tau)=(48,0.1,0.02)$.  ``FNO'' denotes the time-rescaled
  residual-FNO baseline.}
  \label{tab:no-retraining-transfer}
  \small
  \begin{tabular}{ccccc}
    \toprule
    Changed quantity & Value & Potential ETD1 & Direct-force ETD1 & FNO \\
    \midrule
    $\tau$ & $0.01$ & $6.4489\times10^{-4}$ & $1.7997\times10^{-3}$ & $6.4718\times10^{-2}$ \\
    $\tau$ & $0.02$ & $9.8169\times10^{-4}$ & $1.8818\times10^{-3}$ & $6.4580\times10^{-2}$ \\
    $\tau$ & $0.04$ & $1.9400\times10^{-3}$ & $2.5733\times10^{-3}$ & $6.4327\times10^{-2}$ \\
    $\tau$ & $0.08$ & $3.9507\times10^{-3}$ & $4.4450\times10^{-3}$ & $6.3915\times10^{-2}$ \\
    \midrule
    $N$ & $32$ & $9.7633\times10^{-4}$ & $1.8660\times10^{-3}$ & $6.4404\times10^{-2}$ \\
    $N$ & $48$ & $9.8169\times10^{-4}$ & $1.8818\times10^{-3}$ & $6.4580\times10^{-2}$ \\
    $N$ & $64$ & $9.7430\times10^{-4}$ & $1.8775\times10^{-3}$ & $6.5361\times10^{-2}$ \\
    $N$ & $96$ & $9.7310\times10^{-4}$ & $1.8803\times10^{-3}$ & $6.6053\times10^{-2}$ \\
    \midrule
    $\varepsilon$ & $0.05$ & $9.8149\times10^{-4}$ & $1.8593\times10^{-3}$ & $9.4197\times10^{-2}$ \\
    $\varepsilon$ & $0.10$ & $9.8169\times10^{-4}$ & $1.8818\times10^{-3}$ & $6.4580\times10^{-2}$ \\
    $\varepsilon$ & $0.20$ & $8.8985\times10^{-4}$ & $1.7899\times10^{-3}$ & $1.4879\times10^{-1}$ \\
    \bottomrule
  \end{tabular}
\end{table}

The nearly unchanged error over $N=32,48,64,96$ supports grid
portability: the network represents a pointwise constitutive law, while grid
dependence remains in the known discrete operator.  The method also transfers
over the tested range $\varepsilon\in[0.05,0.2]$.  As expected for a
first-order time discretization, the error increases as the deployment time
step grows.  Nevertheless, even at $\tau=0.08$, four times the training time
step, the potential method remains accurate without retraining.

\begin{figure}[htbp]
  \centering
  \includegraphics[width=\linewidth]{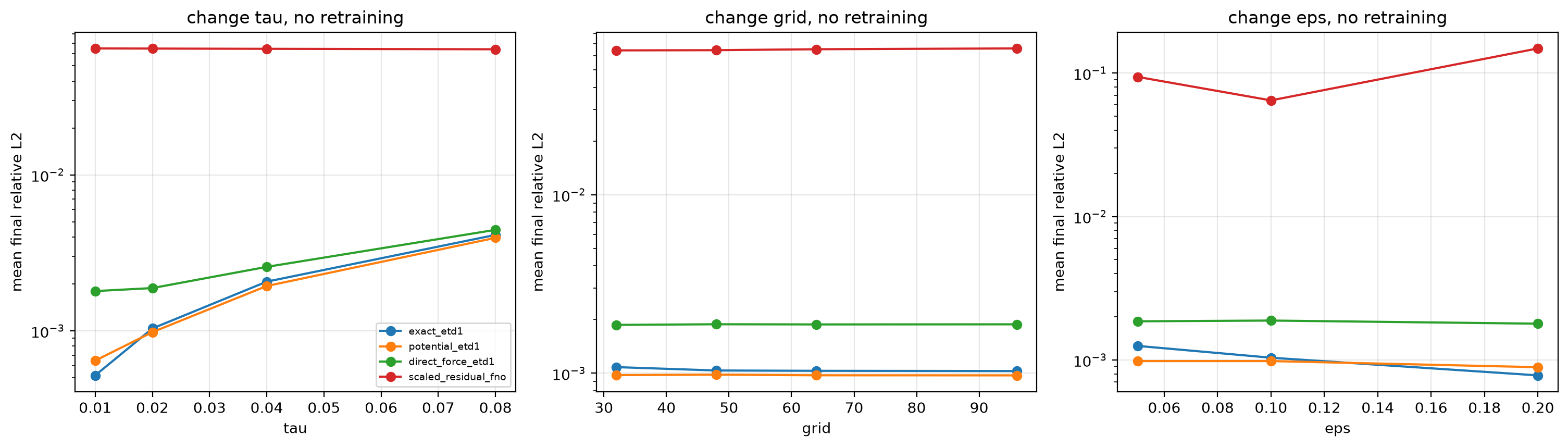}
  \caption{Transfer across time-step sizes, spatial grids, and values of the
  supplied diffusion parameter. Network weights are frozen in every test.}
  \label{fig:no-retraining-transfer}
\end{figure}

\FloatBarrier

\subsection{MBP and unmodified learned-energy dissipation}
\label{subsec:structure-results}

We now use the minimally constrained
architecture~\eqref{eq:endpoint-tanh-potential}. With $\beta=1$, the learned maximum
curvature is $1.938891$, giving $\kappa=1.988891$ after adding a safety margin.
The endpoint derivative errors are
$|F_\theta'(-1)|=|F_\theta'(1)|=2.3842\times10^{-7}$, which is at the
single-precision automatic-differentiation level.  The potential and force
errors are $1.7461\times10^{-3}$ and $5.0619\times10^{-3}$, respectively.

At $T=4$, the final state error is $1.1893\times10^{-3}$ and the maximum
solution magnitude is $0.993598<\beta$.  The largest measured energy increase
is $5.1870\times10^{-7}$, while the largest residual
in~\eqref{eq:checked-energy-inequality} is $1.3286\times10^{-6}$. Both are below
the verification tolerance $2\times10^{-6}$ used for float32 neural
evaluation.  These small positive values are reported rather than clipped to
zero.

\begin{figure}[htbp]
  \centering
  \includegraphics[width=0.92\linewidth]{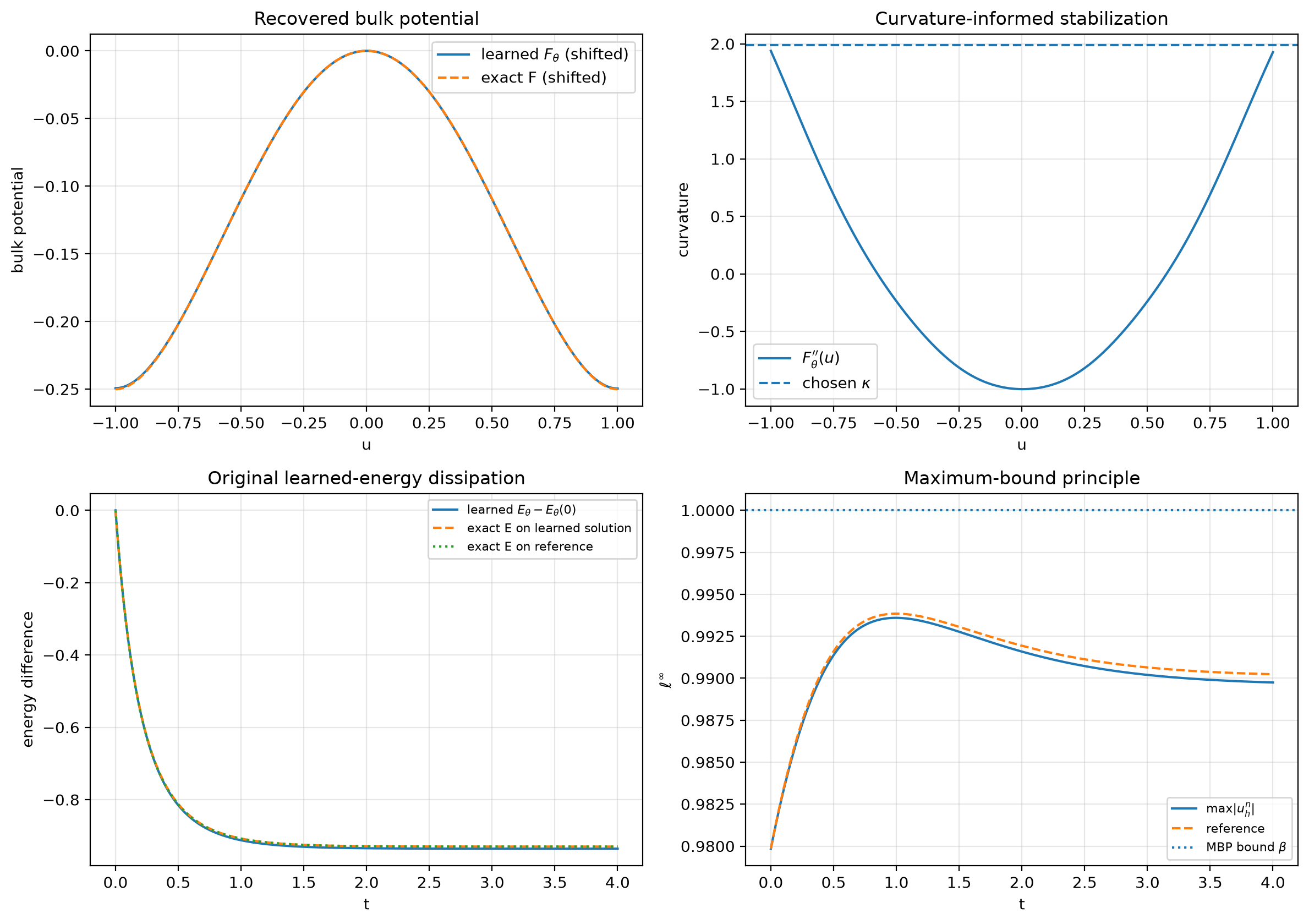}
  \caption{Diagnostics for the endpoint-constrained general tanh potential:
  recovered potential, learned curvature and stabilization, solution error,
  unmodified learned-energy dissipation, MBP, and the checked energy-inequality
  residual.}
  \label{fig:endpoint-tanh-diagnostics}
\end{figure}

The phase-field comparison in Figure~\ref{fig:endpoint-tanh-phase-fields}
shows that the constrained model also reproduces the interface profile and
phase amplitudes throughout the rollout.  The training endpoint $T=1.2$ is
included explicitly; the later snapshots at $T=2$ and $T=4$ therefore test
temporal extrapolation rather than interpolation within the observation
window.

\begin{figure}[htbp]
  \centering
  \includegraphics[width=\linewidth]{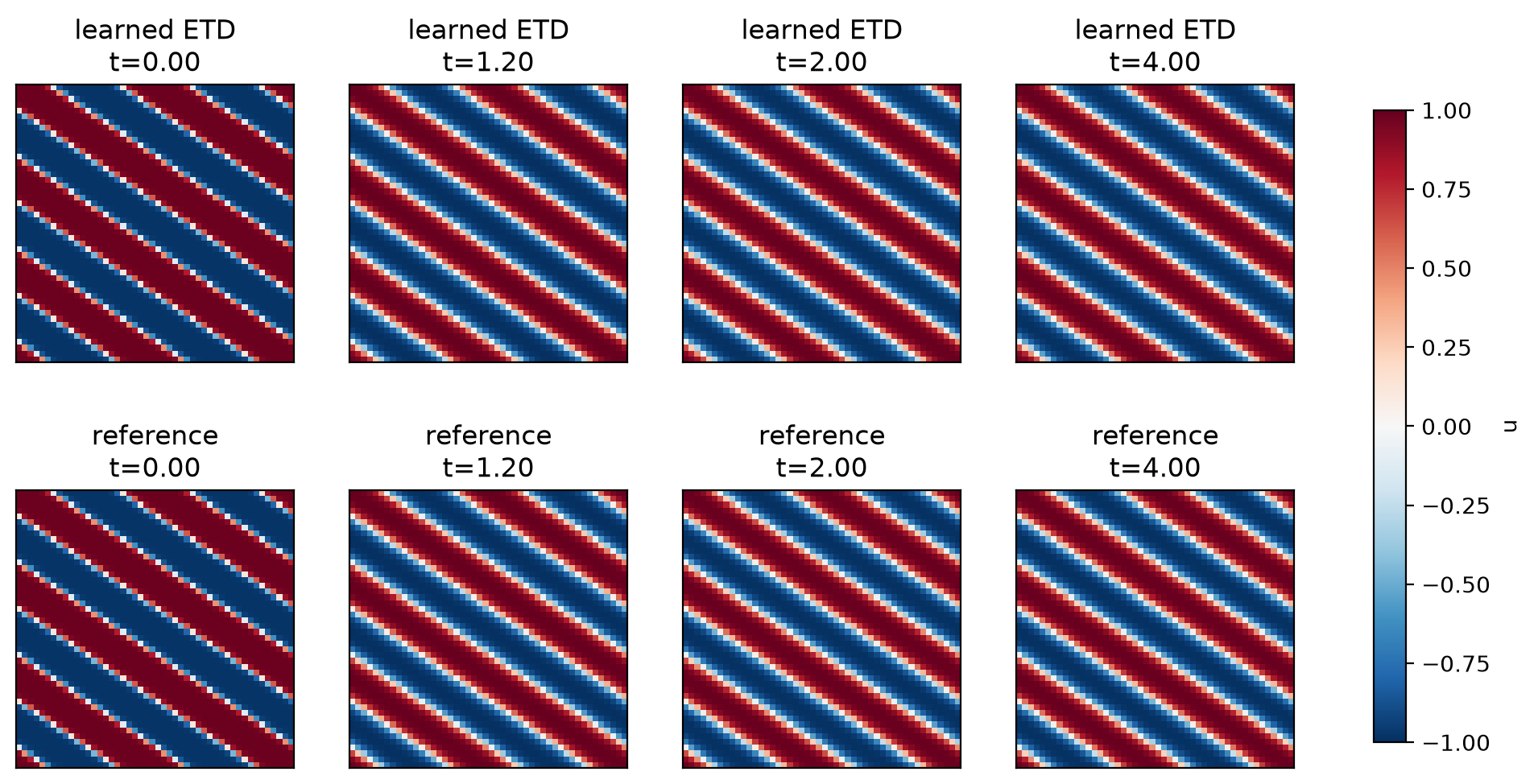}
  \caption{Phase-field comparison for the endpoint-constrained general tanh
  potential.  Top: learned stabilized ETD1.  Bottom: the fine-step reference.
  The common color scale is the prescribed invariant interval $[-1,1]$.}
  \label{fig:endpoint-tanh-phase-fields}
\end{figure}

The same learned potential is reused at four deployment time steps.  As shown
in Table~\ref{tab:tau-stress-structure}, all cases remain finite, satisfy the
MBP, and pass the learned-energy checks.  This is a stress test of structural
robustness, not a temporal-convergence test.  In particular, the nonmonotone
dependence of the state error on $\tau$ may reflect cancellation between the
learned-model error and ETD1 truncation error and should not be interpreted as
higher-order accuracy.

\begin{table}[htbp]
  \centering
  \caption{Time-step stress test at $T=4$ without retraining.  All four cases
  satisfy the MBP and learned-energy monotonicity checks.}
  \label{tab:tau-stress-structure}
  \begin{tabular}{ccccc}
    \toprule
    $\tau$ & Final error & $\max_n\|u_h^n\|_\infty$
    & $\max\Delta E_{\theta,h}$ & Maximum inequality residual \\
    \midrule
    $0.02$ & $1.1893\times10^{-3}$ & $0.993598$
      & $ 5.1870\times10^{-7}$ & $ 1.3286\times10^{-6}$ \\
    $0.05$ & $1.1443\times10^{-3}$ & $0.993618$
      & $-3.5167\times10^{-7}$ & $ 7.9040\times10^{-7}$ \\
    $0.10$ & $1.0694\times10^{-3}$ & $0.993647$
      & $-2.5442\times10^{-6}$ & $ 9.1473\times10^{-7}$ \\
    $0.20$ & $9.8196\times10^{-4}$ & $0.993695$
      & $-1.3740\times10^{-5}$ & $-1.5872\times10^{-6}$ \\
    \bottomrule
  \end{tabular}
\end{table}

\FloatBarrier

\subsection{Five-seed studies of data, noise, and time steps}
\label{subsec:five-seed-study}

We next repeat the endpoint-constrained experiment with five random seeds.
The training horizon is $T_{\rm train}=0.8$ and the test horizon is
$T_{\rm test}=2$.  Table~\ref{tab:data-curve} shows a clear data-scaling
trend.  Increasing the number of state pairs from $160$ to $640$ reduces the
mean rollout error by more than one order of magnitude and reduces the force
error from approximately $7.6\%$ to $0.67\%$.

\begin{table}[htbp]
  \centering
  \caption{Data-volume study over five seeds.  Values are the mean $\pm$
  standard deviation.}
  \label{tab:data-curve}
  \begin{tabular}{ccc}
    \toprule
    State pairs & Rollout error & Learned-force error \\
    \midrule
    $160$ & $1.9304\times10^{-2}\pm4.8926\times10^{-3}$
          & $7.5896\times10^{-2}\pm7.3798\times10^{-3}$ \\
    $320$ & $4.2389\times10^{-3}\pm1.7297\times10^{-3}$
          & $1.8584\times10^{-2}\pm3.3615\times10^{-3}$ \\
    $640$ & $1.5759\times10^{-3}\pm6.0890\times10^{-4}$
          & $6.6559\times10^{-3}\pm8.8430\times10^{-4}$ \\
    \bottomrule
  \end{tabular}
\end{table}

For the noise test, independent Gaussian perturbations are added to both
members of every state pair.  The noise standard deviation is the prescribed
fraction of the root-mean-square magnitude of the corresponding state.
Table~\ref{tab:noise-curve} indicates that $0.1\%$ relative noise has only a
minor influence on the rollout error.  The error increases progressively at
$0.5\%$ and $1\%$ noise, but all tested rollouts remain stable and retain the
MBP and learned-energy monotonicity.

\begin{table}[htbp]
  \centering
  \caption{Noise robustness over five seeds using $640$ state pairs.}
  \label{tab:noise-curve}
  \begin{tabular}{ccc}
    \toprule
    Relative noise & Rollout error & Learned-force error \\
    \midrule
    $0$ & $1.5759\times10^{-3}\pm6.0890\times10^{-4}$
        & $6.6559\times10^{-3}\pm8.8430\times10^{-4}$ \\
    $0.001$ & $1.6554\times10^{-3}\pm6.8687\times10^{-4}$
        & $8.6641\times10^{-3}\pm8.5602\times10^{-4}$ \\
    $0.005$ & $3.1193\times10^{-3}\pm1.1606\times10^{-3}$
        & $1.5001\times10^{-2}\pm5.8227\times10^{-3}$ \\
    $0.010$ & $9.9331\times10^{-3}\pm6.3631\times10^{-3}$
        & $4.3842\times10^{-2}\pm2.4085\times10^{-2}$ \\
    \bottomrule
  \end{tabular}
\end{table}

Figure~\ref{fig:noise-robustness} isolates the noise study from the composite
benchmark plot.  The perturbations are applied independently to both states
in each training pair, whereas evaluation is performed on a clean, unseen
trajectory.  Thus, the reported rollout error measures how noise in the
short-time observations propagates through potential identification and the
subsequent long-time solver, rather than merely measuring reconstruction of a
noisy test field.

\begin{figure}[htbp]
  \centering
  \includegraphics[width=0.72\linewidth]{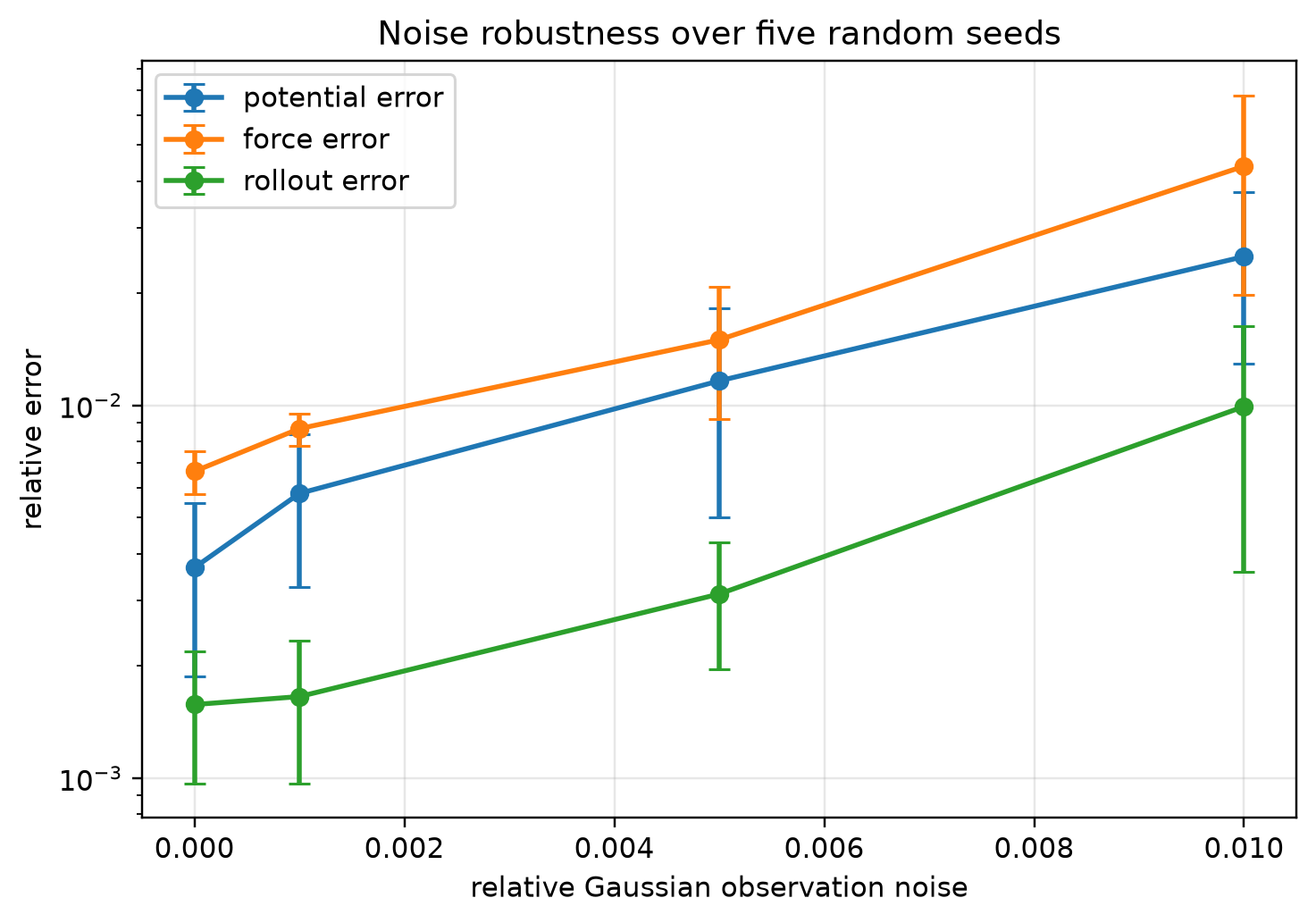}
  \caption{Robustness to independent relative Gaussian noise in the training
  state pairs.  Curves and error bars show the mean and one standard deviation
  over five random seeds.}
  \label{fig:noise-robustness}
\end{figure}

We also compare training at the single step $\tau=0.02$ with pooled training
at $\tau\in\{0.01,0.02,0.04\}$.  The total data budgets are matched
($640$ versus $639$ pairs), and every time-step group is normalized by its own
Duhamel scale.  The pooled model is less sensitive to the deployment step,
but its absolute error is larger; see Table~\ref{tab:multi-tau-transfer}.

\begin{table}[htbp]
  \centering
  \caption{Equal-budget time-step transfer over five seeds.}
  \label{tab:multi-tau-transfer}
  \begin{tabular}{ccc}
    \toprule
    Test $\tau$ & Single-$\tau$ training & Multi-$\tau$ training \\
    \midrule
    $0.01$ & $1.4364\times10^{-3}\pm5.0150\times10^{-4}$
           & $4.8233\times10^{-3}\pm1.4779\times10^{-3}$ \\
    $0.02$ & $1.5759\times10^{-3}\pm6.0890\times10^{-4}$
           & $4.8115\times10^{-3}\pm1.5081\times10^{-3}$ \\
    $0.04$ & $2.2073\times10^{-3}\pm1.0823\times10^{-3}$
           & $4.9404\times10^{-3}\pm1.6998\times10^{-3}$ \\
    $0.08$ & $3.9420\times10^{-3}\pm2.2281\times10^{-3}$
           & $5.6955\times10^{-3}\pm2.4449\times10^{-3}$ \\
    \bottomrule
  \end{tabular}
\end{table}

This negative result is informative.  Under a fixed total budget, pooling
divides the available state coverage among several step sizes.  More
importantly, ETD1 data at different $\tau$ contain different first-order
modified-equation biases, so a single continuous potential must compromise
between inconsistent finite-step targets.  A more suitable extension would
use higher-order ETD observations or Richardson-corrected multi-rate targets.
Across all records in the five-seed study, the maximum endpoint derivative
error is $3.5763\times10^{-7}$, and no MBP or learned-energy monotonicity
failure is observed in either the nominal or time-step-transfer rollouts.

\begin{figure}[htbp]
  \centering
  \includegraphics[width=\linewidth]{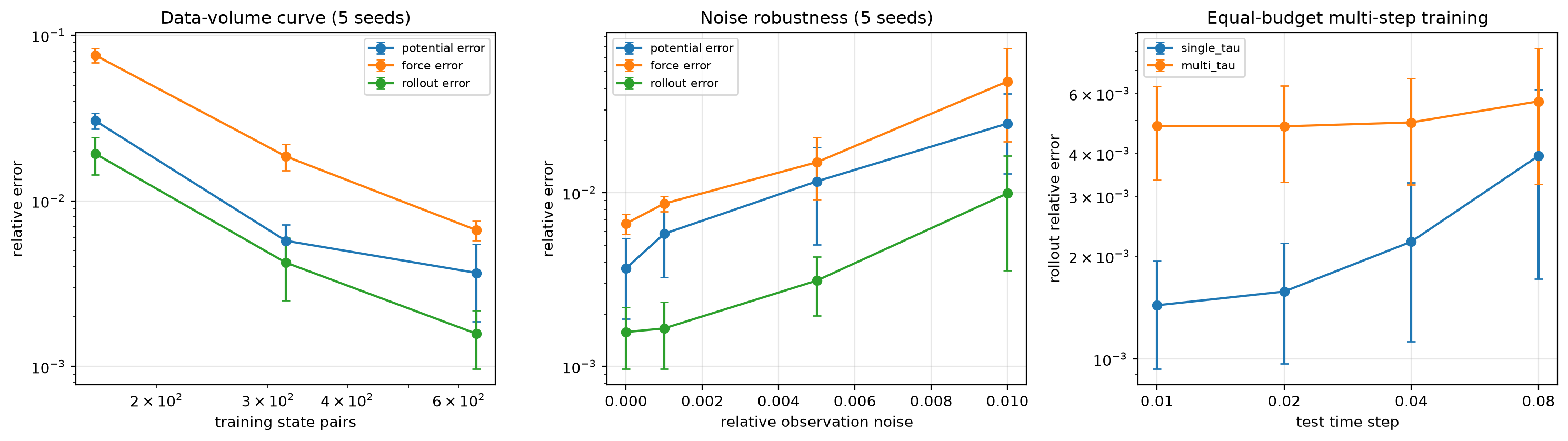}
  \caption{Five-seed benchmark: dependence on the data volume, robustness to
  relative observation noise, and single- versus multi-time-step training.
  Error bars show one standard deviation.}
  \label{fig:five-seed-benchmark}
\end{figure}

\FloatBarrier

\subsection{Non-polynomial Flory--Huggins potential}
\label{subsec:flory-huggins-results}

To test whether the framework is restricted to a polynomial double well, we
consider the logarithmic Flory--Huggins energy
\begin{equation}
  F(u)=\frac{\vartheta}{2}
  \bigl[(1+u)\log(1+u)+(1-u)\log(1-u)\bigr]
  -\frac{\vartheta_c}{2}u^2,
  \qquad |u|<1,
  \label{eq:flory-huggins-potential}
\end{equation}
with $\vartheta=1$, $\vartheta_c=1.5$, and the prescribed separation interval
$|u|\le\beta=0.9$.  Here an even smooth neural potential with hard inward
endpoint-slope signs is used.  This is a problem-specific separation prior and
is stronger than the minimal endpoint-stationarity prior
in~\eqref{eq:endpoint-tanh-potential}. The purpose of this test is to assess
non-polynomial recovery and strict separation rather than to claim that the
logarithmic singularity is learned without structural information.

Only state pairs from $0\le t\le1.2$ are used for training.  The exact
logarithmic potential and force are hidden from the loss and are used only to
generate a fine-step reference and to compute post-training errors.  The
learned model is then advanced to $T=20$ with no SAV or GSAV variable.
Table~\ref{tab:flory-results} summarizes the results.

\begin{table}[htbp]
  \centering
  \caption{Flory--Huggins state-only learning and long-time rollout.}
  \label{tab:flory-results}
  \begin{tabular}{lc}
    \toprule
    Diagnostic & Value \\
    \midrule
    Potential relative error & $3.4388\times10^{-3}$ \\
    Nonlinear-force relative error & $2.4121\times10^{-2}$ \\
    Curvature relative error & $9.4421\times10^{-2}$ \\
    Final rollout error at $T=20$ & $1.2072\times10^{-2}$ \\
    Maximum solution magnitude & $0.852836$ \\
    Distance from the singular boundary $|u|=1$ & $0.147164$ \\
    Maximum learned-energy increment & $-9.5139\times10^{-5}$ \\
    Maximum energy-inequality residual & $-3.1107\times10^{-6}$ \\
    \bottomrule
  \end{tabular}
\end{table}

The learned rollout remains strictly inside $[-0.9,0.9]$, dissipates the
unmodified learned energy, and reaches $T=20$ with a final relative error of
approximately $1.2\%$. The result indicates that the approach is not tied
to the cubic Allen--Cahn force, although curvature recovery becomes more
difficult near a logarithmic singularity.

\begin{figure}[htbp]
  \centering
  \includegraphics[width=0.95\linewidth]{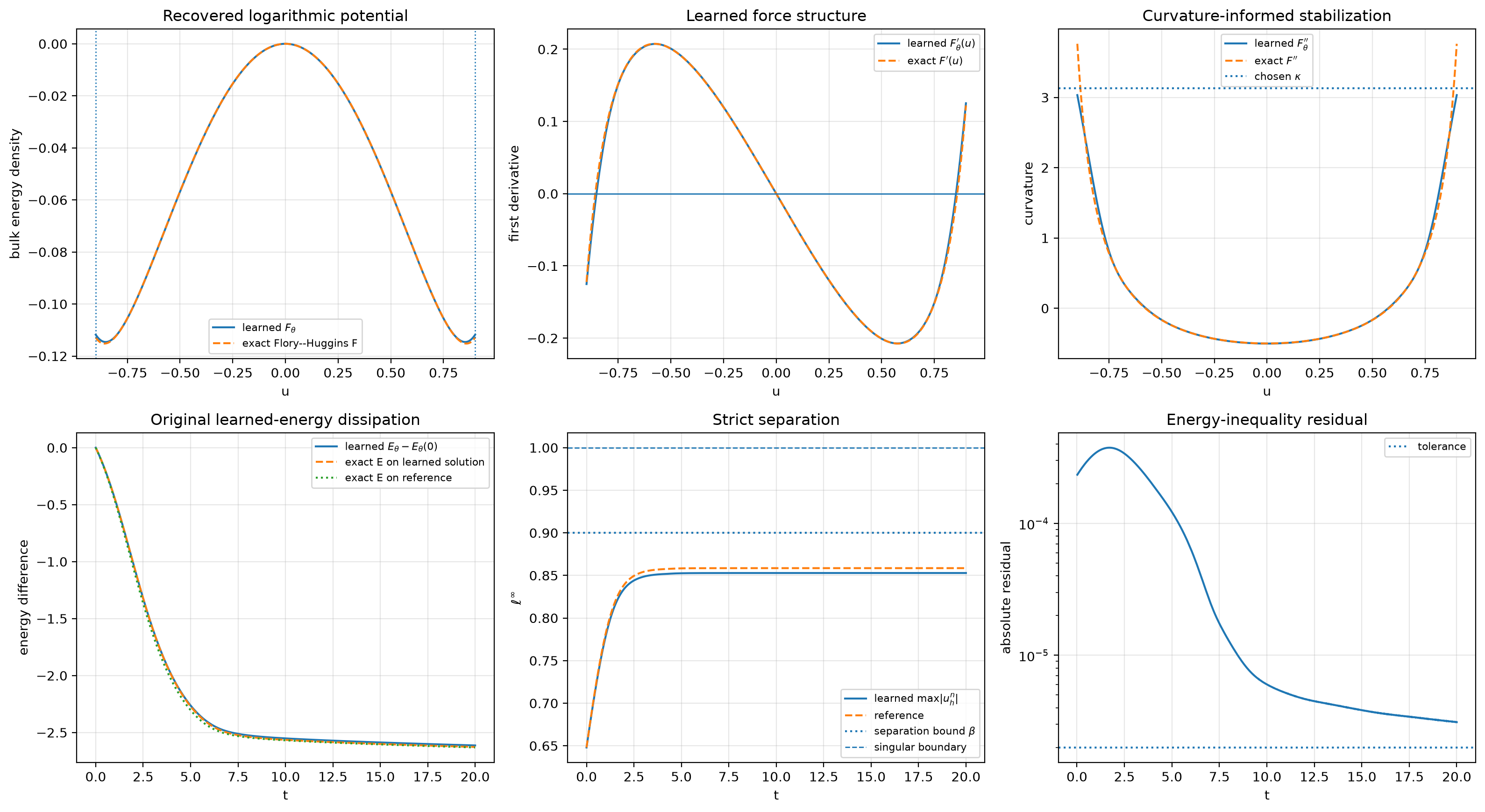}
  \caption{Flory--Huggins diagnostics: recovery of the logarithmic energy and
  force, curvature-informed stabilization, long-time state error, strict
  separation, and dissipation of the unmodified learned energy.}
  \label{fig:flory-huggins-diagnostics}
\end{figure}

Figure~\ref{fig:flory-huggins-phase-fields} further shows that the learned
model reproduces the phase-selection and coarsening pattern of the logarithmic
reference dynamics.  The agreement at $T=20$, compared with the training
horizon $T_{\rm train}=1.2$, supports the use of the identified non-polynomial
potential for long-time phase-field simulation.

\begin{figure}[htbp]
  \centering
  \includegraphics[width=0.88\linewidth]{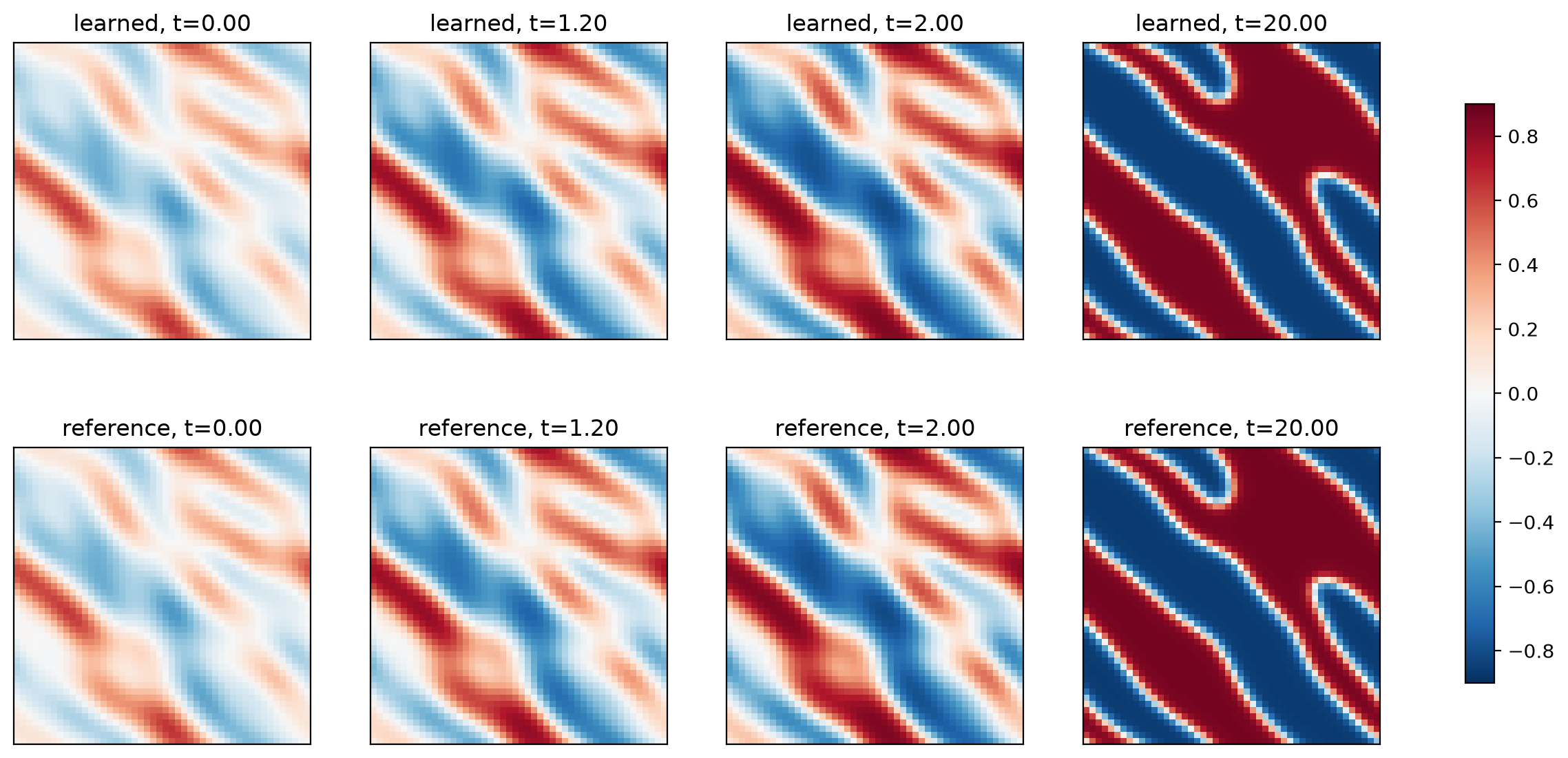}
  \caption{Flory--Huggins phase-field evolution.  Top: learned-potential ETD1.
  Bottom: the fine-step exact-potential reference.  The learned solution
  remains inside the strict separation interval and reproduces the reference
  phase morphology up to $T=20$.}
  \label{fig:flory-huggins-phase-fields}
\end{figure}

\clearpage

\subsection{Summary of the observed advantages}
\label{subsec:observed-advantages}

The numerical evidence supports the following conclusions within the tested
periodic phase-field setting.
\begin{enumerate}
  \item \emph{Low-dimensional mechanism identification.}
  The method learns a scalar potential rather than a high-dimensional
  space--time solution or a full state-update operator.  This yields accurate
  force recovery from short state trajectories and an interpretable learned
  energy.

  \item \emph{Separation of known and unknown physics.}
  The known stiff diffusion operator is retained in the ETD solver.  The
  neural network is used only for the missing constitutive term, which avoids
  spending model capacity on relearning the known spatial dynamics.

  \item \emph{Long-time and cross-configuration reuse.}
  A potential trained up to $T=1.2$ remains accurate and stable as far as
  $T=50$.  The same weights transfer to unseen initial conditions, grids from
  $32^2$ to $96^2$, diffusion parameters from $0.05$ to $0.2$, and time steps
  up to four times the training step.

  \item \emph{Structure preservation without an auxiliary energy.}
  The curvature-informed deployments satisfy the tested MBP and original
  learned-energy dissipation diagnostics.  No SAV/GSAV variable or modified
  auxiliary energy is introduced.

  \item \emph{Reproducibility and data scaling.}
  The five-seed results show a consistent reduction in both force and rollout
  errors as the number of state pairs increases.  The method is insensitive
  to $0.1\%$ relative noise and degrades gradually at higher noise levels.

  \item \emph{Advantage over matched-size black-box learning.}
  In the present benchmark, the potential representation is more accurate
  than a parameter-matched direct-force network and substantially more robust
  than a residual FNO baseline. This supports the use of a structure-aware gray-box
  representation when the differential operator is known.
\end{enumerate}

These experiments establish numerical feasibility and identify the regime in
which the method is most compelling: the differential operator is known, the
local potential is unknown, short-time state observations are available, and
many long-time or multi-initial-condition predictions are required.  They do
not by themselves prove superiority over every PINN or neural-operator
variant, nor do they establish a computational advantage for a single forward
solve when the complete potential is already known.  Such claims require
matched wall-clock, memory, and accuracy comparisons and are left for future
work.

\section{Conclusions and extensions}
We have developed a gradient-flow framework based on a trajectory-inferred
free energy, in which the local free-energy density is identified directly
from observed state transitions while the known structure of the linear
spatial operator is retained. In addition to the fixed-operator setting, we
formulated joint identification of a positive scalar Laplacian coefficient
$a_\alpha$, representing $\varepsilon^2$, and the local energy $F_\theta$.
The central construction is constitutive-first rather than flow-map-first:
the data identify $(a_\alpha,F_\theta)$, differentiation produces the
variational force, and the coefficient-dependent ETD response then converts
the inferred energy into a recurrent state-update map.  This differs from
directly learning the full map
$U^n\mapsto U^{n+1}$ with a neural operator.
The inferred curvature determines the stabilization in the resulting
\TIFE-based ETD1 map. Under the stated semigroup, endpoint, and upper-curvature
curvature assumptions, the method preserves the admissible maximum bound and
unconditionally dissipates the original discrete \TIFE{} functional, without
introducing an SAV/GSAV auxiliary variable or a modified energy.  The error
analysis separates temporal, spatial, and force-inference effects and yields
the bound \(O(\tau+h^2+\delta_F)\) for a fixed operator and
\(O(\tau+h^2+\delta_F+\delta_a)\) for the inferred coefficient, where
$\delta_a=|a_\alpha-a_\star|$. In addition, the
response-to-force estimate bounds $\delta_F$ on the data-supported
interval by the achieved ETD response residual, observation perturbations,
the finite-step freezing defect, and the state fill distance, with explicit
dependence on response-inversion conditioning.  Under the stated joint
identifiability condition, the coefficient-dependent response residual also
controls $\delta_a$ and the empirical force error.  This closes a conditional
data-to-rollout error chain rather than treating the inferred constitutive
quantities as unrelated external errors.

The numerical results complement the analysis.  State-only training recovers
both quartic and logarithmic potentials, produces accurate long-time rollouts
well beyond the observation window, and transfers without retraining across
initial conditions, grids, diffusion parameters, and time steps.  The
curvature-informed deployments satisfy the tested maximum-bound and
learned-energy diagnostics, while the five-seed studies quantify data scaling
and robustness to observation noise.  The matched-size comparisons also show
the benefit, for the present gray-box problem, of learning a scalar
constitutive energy rather than a full update operator.
The reported numerical tables use a supplied Laplacian coefficient and hence
do not yet constitute an empirical validation of joint coefficient recovery.

Natural extensions include second-order \TIFE-based ETD schemes, certified
curvature bounds for neural representations, vector-valued order parameters
with Hessian-based stabilization, jointly inferred multi-parameter elliptic
operator families, and
nonlocal or mobility-dependent free energies.

\bigskip
\textbf{Acknowledgements.} G. Ji is partially supported by the National Natural
Science Foundation of China (Grant No. 12471363).

\nocite{ball2011orientability}
\bibliographystyle{siam}
\bibliography{S0362546X14002934}

\end{document}